\documentclass[11pt,reqno]{amsart}
\usepackage{}
\usepackage{a4wide}
\numberwithin{equation}{section}
\usepackage{mathrsfs}
\usepackage{amsfonts}
\usepackage{amsmath}
\usepackage{stmaryrd}
\usepackage{amssymb}
\usepackage{amsthm}
\usepackage{mathrsfs}
\usepackage{url}
\usepackage{amsfonts}
\usepackage{amscd}
\usepackage{indentfirst}
\usepackage{enumerate}
\usepackage{amsmath,amsfonts,amssymb,amsthm}
\usepackage{amsmath,amssymb,amsthm,amscd}
\usepackage{graphicx,mathrsfs}
 \usepackage{appendix}

 \usepackage{csquotes}
 \usepackage{ulem} 
 \usepackage{cancel}

\usepackage[colorlinks,linkcolor=blue, citecolor=blue]{hyperref}
\numberwithin{equation}{section}

\usepackage[numbers,sort&compress]{natbib}

\usepackage{esint}
\usepackage{graphicx}
\usepackage[dvipsnames]{xcolor}
\newcommand{\R}{\mathbb{R}}

\newcommand{\N}{\mathbb{N}}

\newcommand\e\varepsilon
\newtheorem{Thm}{Theorem}[section]
\newtheorem{Lem}{Lemma}[section]
\newtheorem{Prop}{Proposition}[section]
\newtheorem{Def}{Definition}[section]

\newtheorem{Rem}{Remark}[section]

\begin{document}

\title[Nonlinear Grushin equation  ]
{ On a class of Nonlinear Grushin {equations} }

\author[W. Bauer, Y. Wei  and X. Zhou]{Wolfram Bauer, Yawei Wei  and Xiaodong Zhou}

\address[Wolfram Bauer]{Leibniz Universit\"{a}t Hannover, Institut f\"{u}r Analysis, Welfengarten 1, 30167
Hannover, Germany}
\email{bauer@math.uni-hannover.de, bauerwolfram@web.de}

\address[Yawei Wei]{School of Mathematical Sciences and LPMC, Nankai University, Tianjin 300071, China}
\email{weiyawei@nankai.edu.cn}

\address[Xiaodong Zhou]{School of Mathematical Sciences, Nankai University, Tianjin 300071, China}
\email{1120210030@mail.nankai.edu.cn}

\thanks{
affiliations: {\bf Wolfram Bauer} {\bf (corresponding author}, ORCID: 0000-0001-5368-8708, {\it Institut f\"{u}r Analysis, Welfengarten 1, 30167 Hannover, Germany}, {\bf email} \texttt{bauer@math.uni-hannover.de}; {\bf Yawei Wei} ORCID: 0000-0002-9743-917X, {\it School of Mathematical Sciences and LPMC, Nankai University, Tianjin 300071, China} {\bf email:} \texttt{weiyawei@nankai.edu.cn}; {\bf Xiaodong Zhou}, ORCID: 0009-0000-7962-6521, {\it School of Mathematical Sciences, Nankai University, Tianjin 300071, China} {\bf email:} \texttt{1120210030@mail.nankai.edu.cn}; \\
Acknowledgements: This work is supported by the NSFC under the grands 12271269 and the Fundamental Research Funds for the Central Universities.}





\begin{abstract} In this paper, we study {two} kinds of nonlinear degenerate elliptic equations containing the Grushin operator. First, we prove radial symmetry and a decay rate at infinity of solutions to such a Grushin equation by using the moving plane method in combination with suitable integral inequalities. Applying similar methods, we obtain nonexistence results for solutions to a second type of Grushin equation in Euclidean half space. Finally, we derive a priori estimates {and the existence} for positive solutions to more general types of Grushin equations by employing blow up analysis and topological degree methods, respectively.
\end{abstract}

\maketitle
\section{Introduction}

\setcounter{equation}{0}

In this paper, we establish a priori estimates and prove existence results for positive solutions to the following semilinear degenerate elliptic equation
\begin{equation}\label{002.1}
\begin{cases}
-\Delta_\gamma u=f(z,u),\;\;u>0,&~ \mbox{in} ~\Omega,\\[2mm]
u=0,  &~ \mbox{on} ~\partial\Omega,
\end{cases}
\end{equation}
where $\Omega\subset\R^{N+l}$ is a bounded open domain containing the origin and with $C^1$ boundary $\partial\Omega$. The right hand side
$f(z,u):\Omega\times\R\rightarrow\R$ is a continuous function satisfying some technical conditions specified below.
Here, $\Delta_\gamma$ is the well-known $Grushin\; operator$ given by
\begin{equation}\label{005.250}
  \Delta_\gamma u(z)=\Delta_xu(z)+|x|^{2\gamma}\Delta_yu(z),
\end{equation}
where, $\Delta_x$ and $\Delta_y$ denote the Laplace operators in the variables $x$ and $y$, respectively, with $z=(x,y)\in\R^{N+l}$ and $N+l\geq3$. Moreover,
$\gamma>0$ is a real number and  we write
\begin{equation}\label{wys3}
  N_\gamma=N+(1+\gamma)l
\end{equation}
for the associated homogeneous dimension. In some of our results we assume that $\gamma \in \mathbb{N}$ is an integer in which case $\Delta_{\gamma}$ is a H\"{o}rmander operator.

As a first step in the analysis we prove a symmetry result by using the moving plane method for the following equation:
\begin{equation}\label{0}
    \begin{cases}
-\Delta_\gamma u+ u=u^p,\;\;u>0,&~ \mbox{in} ~\R^{N+l},\\[2mm]
u(z)\rightarrow0,\;\;&~ \mbox{as} ~d(z,0)\rightarrow\infty,
\end{cases}
\end{equation}
where $p>1$, and $d(z,0)$ is a suitably defined distance on $\R^{N+l}$, namely
\begin{equation}\label{006.6}
  d(z,0):=\left(\frac{1}{(1+\gamma)^2}|x|^{2+2\gamma}+|y|^2\right)^{\frac{1}{2+2\gamma}}.
\end{equation}
Then we provide two nonexistence results by applying the moving plane method to the following problems:
\begin{equation}\label{32}
    \begin{cases}
-\Delta_\gamma u=u^p,\;\; u>0,&~ \mbox{in} ~\R^{N+l},\\[2mm]
u(z)\rightarrow0,\;\;&~ \mbox{as} ~d(z,0)\rightarrow\infty,
\end{cases}
\end{equation}
where {$1<p<\frac{N_\gamma+2}{N_\gamma-2}$}, and
\begin{equation}\label{004.3}
\begin{cases}
-\Delta_\gamma u=u^p,\;\;u>0,&~ \mbox{in} ~\R^{N+l}_+,\\[2mm]
u=0,  &~ \mbox{on} ~\partial\R^{N+l}_+,
\end{cases}
\end{equation}
where {$1+\frac{4}{N_\gamma}\leq p<\frac{N_\gamma+2}{N_\gamma-2}$}. Throughout the paper $\R^{N+l}_+$ denotes the half space, i.e.
\begin{equation}\label{004.1}
  \R^{N+l}_+=\{z=(x,y)\in\R^{N+l}|\;z=(x,\tilde{y}^{\prime},y_l), y_l>0\}, \hspace{3ex} \mbox{\it with} \hspace{3ex}  \tilde{y}'=(y_1,\cdots,y_{l-1}).
\end{equation}
We mention that the nonexistence result for \eqref{32} is known and directly follows from \cite[Theorem 1.1]{17} by Yu.
Establishing a priori estimates for positive solutions to \eqref{002.1} leads to Liouville type theorems for
\eqref{32} and \eqref{004.3} by blow up analysis. A proof of the existence of positive solutions to \eqref{002.1} essentially relies on these a priori estimate combined with  topological degree methods.

The Grushin operator $\Delta_{\gamma}$ was first studied by Grushin in \cite{6} and \cite{7}. It is worth noting that $\Delta_\gamma$ is elliptic for $x\neq0$ and degenerates on the manifold $\{0\}\times\R^l$. When $\gamma\in\N$, $\Delta_{\gamma}$ is a H\"{o}rmander operator; furthermore when $\gamma>0$ is an even integer, the vector fields $X_1=\frac{\partial}{\partial x_1}$, $\cdots$, $X_N=\frac{\partial}{\partial x_N}$, $Y_1=|x|^{\gamma}\frac{\partial}{\partial y_1}$, $\cdots$, $Y_l=|x|^{\gamma}\frac{\partial}{\partial y_l}$ are smooth and satisfy H\"ormander's condition, which implies hypoellipticity of $\Delta_\gamma$. Geometrically, $\Delta_\gamma$ is induced from a sub-Laplace operator on a nilpotent Lie group of step $\gamma+1$ (for $\gamma \in \mathbb{N}$) by a submersion. Specific explanation
of the geometric framework can be found in \cite{Bauer1} by Bauer et al. Abatangelo et al. \cite{yin} used a blow-up monotonicity argument to recover the asymptotic behavior of solutions to the linear Grushin equation with a potential, and they also studied the Grushin operator in spherical coordinates. The existence results for nonlinear degenerate elliptic Grushin type equations such as \eqref{0} were given in \cite{4}.

This type of equation may have no variational structure. Therefore we exploit techniques from the theory of nonlinear operators rather than calculus of variations. More specifically, we derive a priori estimates for positive solutions to the nonlinear problem \eqref{002.1}. Blow up analysis is essentially used, and finally this problem is reduced to the proof of Liouville type theorems in
$\R^{N+l}$ and $\R^{N+l}_+$. We emphasize that the derivation of Liouville type theorems corresponding to the above problems is a common technique and of great importance.
During the last decades various authors have contributed to this field. Gidas and his collaborator initially introduced such a procedure in the
analysis of certain superlinear elliptic equations in \cite{9}. Later Chang included it into his textbook \cite{8}. For the following equation:
\begin{equation}\label{005.3}
\begin{cases}
  \displaystyle\sum_{i,j=1}^n\frac{\partial}{\partial x_i}\left(a_{ij}(x)\frac{\partial u}{\partial x_j}\right)+\sum_{j=1}^nb_j(x)\frac{\partial u}{\partial x_j}+f(x,u)=0,\;\;&~ \mbox{in} ~\Omega,\\[2mm]
u=\phi,\;\;&~ \mbox{on} ~\partial\Omega,
\end{cases}
\end{equation}
in which $\Omega\subset\R^n$ is an open bounded domain with smooth boundary, $f\in C(\overline{\Omega}\times\R)$ is superlinear in $u$, $a_{ij}(x)$, $b_j(x)\in C(\overline{\Omega})$, $i,j=1,\cdots,n$, $\phi\in C^1(\partial\Omega)$ is nonnegative, the authors provide a priori bounds for positive solutions and then apply degree theory to obtain the solutions to \eqref{005.3}. They introduced a useful method - the scaling method - based on Liouville type theorems for the equations:
\begin{equation}\label{005.4}
  -\Delta u=u^p,\;\;u>0,\;\; ~\mbox{in}~\R^n,
\end{equation}
and
\begin{equation}\label{005.5}
\begin{cases}
-\Delta u=u^p,\;\;&~ \mbox{in} ~\R^n_+,\\[2mm]
u=0,\;\;&~ \mbox{on} ~\partial\R^n_+,
\end{cases}
\end{equation}
where $p\in(1,\frac{n+2}{n-2})$, $n\geq3$, $\R^n_+=\{x\in\mathbb{R}^n\:|\;x=(x_1, x_2,\cdots, x_n), x_n>0\}$.
Liouville type theorems for \eqref{005.4} and \eqref{005.5} were studied by Gidas and his collaborator in \cite{5} and \cite{9}, respectively. Finally, it has been observed that Equation \eqref{005.3} has at least one positive solution based on the Leary Schauder degree theorem in \cite{10}.

Other relevant examples are nonlinear equations involving nonlocal elliptic operators including the fractional Laplacian and have been studied by Chen together with his collaborators in \cite{11}. More precisely, the following problem is considered:

\begin{equation}\label{005.6}
\begin{cases}
(-\Delta)^{\frac{\alpha}{2}}u(x)=u^p(x),\;\;& x\in\Omega,\\[2mm]
u(x)\equiv0,\;\;& x\notin\Omega,
\end{cases}
\end{equation}
where
\begin{equation}\label{005.7}
\begin{aligned}
  (-\Delta)^{\frac{\alpha}{2}}u(x)
  &=C_{n,\alpha}P.V.\int_{\R^n}\frac{u(x)-u(z)}{|x-z|^{n+\alpha}}dz
\\&=C_{n,\alpha}\lim_{\varepsilon\rightarrow0^+}\int_{\R^n\setminus B_\varepsilon(x)}\frac{u(x)-u(z)}{|x-z|^{n+\alpha}}dz
\end{aligned}
\end{equation}
is a nonlocal operator, $0<\alpha<2$, $1<p<\frac{n+\alpha}{n-\alpha}$ and $\Omega\subset\R^n$ is a bounded domain with $C^2$ boundary.
As before, the authors develop a direct blow-up and rescaling argument to obtain a priori estimates for \eqref{005.6}. This problem finally is reduced to the proof of Liouville type theorems for the equations:
\begin{equation}\label{005.8}
(-\Delta)^{\frac{\alpha}{2}}u(x)=u^p(x),\;\;u(x)>0,\;\; x\in\R^n,
\end{equation}
and
\begin{equation}\label{005.9}
\begin{cases}
(-\Delta)^{\frac{\alpha}{2}}u(x)=u^p(x),\;\;& x\in\R^n_+,\\[2mm]
 u(x)=0,\;\;& x\notin\R^n_+.
\end{cases}
\end{equation}
Topological degree theories are introduced that are necessary to derive the existence results and in \cite{11} various forms are explained in detail. We only give the simplest example in order to
demonstrate the method. Liouville type theorems on $\R^n$ and $\R^n_+$ were studied by Chen and his collaborator in \cite{12} and \cite{13}, respectively.

It is crucial that a Liouville type theorem for the corresponding limit problem induces a priori estimates. It should also be mentioned that the moving plane method is a powerful tool for obtaining Liouville type theorems and has been applied previously by many authors. Guo and Liu in \cite{14} studied Liouville type theorems for positive solutions to a general classical system of equations. The key tool in \cite{14} is the moving plane method combined with integral inequalities. Similar results exist in connection with the Grushin operator
$\Delta_{\gamma}$ above. In \cite{15} Wang proved a Liouville type theorem
for positive weak solutions to a nonlinear system containing $\Delta_{\gamma}$. As in \cite{14}, a key tool in \cite{15} is the moving plane method together with integral inequalities.

For a single equation containing the Grushin operator related aspects have been studied. Monticelli in \cite{1} established a maximum principle for a class of linear, degenerate elliptic differential operators of second order. In particular, this applies to the Grushin operator. Based on such a maximum principle the author proved a nonexistence result for classical solutions to Grushin equations on Euclidean space and with subcritical growth. Monti et al. \cite{zs4} studied positive solutions to critical semilinear Grushin type equations via the Kelvin transformation. Yu in \cite{17} obtained the nonexistence of positive solutions in Euclidean space for an elliptic equation involving the Grushin operator with a general nonlinearity. Here we obtain a nonexistence result in case of the Grushin equation in Euclidean half space with nonlinear term of the form $u^p$.

Special attention to problems involving the Grushin operator is also given in \cite{4} and \cite{41}, Monti \cite{3}, Monticelli and Payne \cite{22}, Tri \cite{Tri4} and \cite{Tri5}, Metafune et al. \cite{Lp12}.
\vspace{1ex}\par
In order to explicitly state our main results we now introduce some notations and preliminaries that will be used throughout the paper.
We define the weighted Sobolev space 
\begin{equation}\label{006.2}
  H^{1,2}_\gamma(\R^{N+l})=\left\{u\in L^2(\R^{N+l})\; |\;\frac{\partial u}{\partial x_i}, |x|^\gamma\frac{\partial u}{\partial y_j}\in L^2(\R^{N+l}), i=1,\cdots,N, j=1,\cdots,l\right\}.
\end{equation}
Note that $H^{1,2}_\gamma(\R^{N+l})$ endowed with the inner product:
\begin{equation}\label{006.3}
  \langle u,v\rangle_\gamma=\int_{\R^{N+l}}\nabla_\gamma u \cdot \nabla_\gamma v+uvdz,
\end{equation}
is a Hilbert space. Here we write $\nabla_{\gamma}$ for the gradient $\nabla_{\gamma}:=(\partial_{x_1}, \ldots, \partial_{x_N}, |x|^\gamma\partial_{y_1}, \ldots, |x|^\gamma\partial_{y_l})$. Moreover, the norm on $H^{1,2}_\gamma(\R^{N+l})$ induced by \eqref{006.3} is denoted by:
\begin{equation}\label{006.4}
  \|u\|_\gamma=\Big(\int_{\R^{N+l}}|\nabla_\gamma u|^2+|u|^2dz\Big)^\frac{1}{2}.
\end{equation}
Occasionally we write $\| \cdot\|_{H_{\gamma}^{1,2}(\Omega)}$ instead of $\| \cdot \|_{\gamma}$ where $\Omega$ is a domain in $\mathbb{R}^{N+l}$.
The above definitions can be found in \cite{4} and \cite{2} and according to the Sobolev inequality studied by Monti in \cite{3}:
\begin{equation}\label{006.5}
  \left(\int_{\R^{N+l}}|u|^{2^\ast_\gamma}dz\right)^{\frac{2}{2^\ast_\gamma}}\leq C(N,l,\gamma)\int_{\R^{N+l}}|\nabla_\gamma u|^2dz,
\end{equation}
we conclude that the embedding $H^{1,2}_\gamma(\R^{N+l})\hookrightarrow L^{2^\ast_\gamma}(\R^{N+l})$ with $2^\ast_\gamma=\frac{2N_\gamma}{N_\gamma-2}$ is continuous. When $\Omega\subset\R^{N+l}$ is a bounded domain, $H^{1,2}_\gamma(\Omega)$ can be defined similarly. Furthermore,
\begin{equation}\label{two-norms-GL}
\Big(\displaystyle\int_{\Omega}|\nabla_\gamma u|^2+|u|^2dz\Big)^\frac{1}{2} \hspace{3ex}  \mbox{ and} \hspace{3ex}
\Big(\displaystyle\int_{\Omega}|\nabla_\gamma u|^2dz\Big)^\frac{1}{2}
\end{equation}
are equivalent norms in $H^{1,2}_\gamma(\Omega)$ when $\Omega\subset\R^{N+l}$ is a bounded domain.
Let $d(z,0)$ be the distance on~$\R^{N+l}$ defined in (\ref{006.6}).
For~$z=(x,y)\in\R^{N+l}$ and $r>0$ we denote by
\begin{equation}\label{006.7}
  \widetilde{B}_r(0):=\{z=(x,y)\in\R^{N+l}|\;d(z,0)<r\}
\end{equation}
the open $r$-ball with respect to $d$.

Next,  we recall a maximum principle for a class of linear, degenerate elliptic differential operators of second order generalizing the Grushin operator.
For details we refer to \cite{1}.

Let $\Omega\subset\R^N$ be a bounded and connected domain, and consider the linear differential operator $L$ in $\Omega$:
\begin{equation}\label{006.8}
  Lu:=\sum_{i,j=1}^Na_{ij}(x)D_{ij}u+\sum_{i=1}^Nb_i(x)D_iu+c(x)u
\end{equation}
for $u\in C^2(\Omega)\cap C(\overline{\Omega})$. Assume that $b_i, c \in L^\infty(\Omega)$, {$a_{ij}\in C(\overline{\Omega})$, and $a_{ij}=a_{ji}$ for $i,j=1,\cdots,N$, and $L$ has nonnegative characteristic form in $\Omega$, i.e.
\begin{equation}\label{006.9}
  \sum_{i,j=1}^Na_{ij}(x)\xi_i\xi_j\geq0,\;\;\;\forall x\in\Omega,\;\forall\xi\in\R^N.
\end{equation}
Moreover, let $A(x):=[a_{ij}(x)]$ be a real symmetric matrix and let $\Sigma:=\{x\in\overline{\Omega}\;|\;\det A(x)=0\}$ be the degeneracy set of the operator $L$.
Moreover, assume that $L$ satisfies the conditions $(E_{\xi})$ and $(\Sigma)$ below.
For the convenience of the reader we repeat the definitions in \cite[p. 619]{1} (see \cite[Rem 2.9]{1} for a geometric interpretation):

\vspace{1mm}
\begin{itemize}
\item[$(E_\xi)$] There exist $\beta>0$, $\xi\in\R^N$ with $|\xi|=1$, such that for any $x\in\Omega$, we have
\begin{equation*}
  \langle\xi,A(x)\xi\rangle_{\R^N}:=\sum_{i,j=1}^Na_{ij}(x)\xi_i\xi_j\geq\beta>0.
\end{equation*}
\item[$(\Sigma)$] The following conditions are fulfilled: \vspace{1mm}
\begin{itemize}
\item $\Sigma$ has no interior points. We let $\Omega_1, \Omega_2,\cdots$ denote the at most countably many connected components of $\Omega\setminus\Sigma$.
\item  $\Sigma\cap\Omega=\Sigma_1\cup\Sigma_2$, where for all $x_0\in\Sigma_1$ and $\Omega_m$ such that $x_0\in\partial\Omega_m$, there is
$\overline{B_r(x_1)}\subset\overline{\Omega_m}$ such that $x_0\in\partial B_r(x_1)$
\begin{equation*}
 \big{\langle}(x_0-x_1), A(x_0)(x_0-x_1)\big{\rangle}_{\R^N}>0 \hspace{3ex} \mbox{\it and } \hspace{3ex} \overline{B_r(x_1)}\cap\Sigma=\{x_0\}.
\end{equation*}
while for all $x_0\in\Sigma_2$ there exists $\overline{B_r(x_1)}\subset\Omega$ such that $x_0\in\partial B_r(x_1)$
\begin{equation*}
\big{\langle}(x_0-x_1), A(x_0)(x_0-x_1)\big{\rangle}_{\R^N}>0\hspace{3ex} \mbox{\it and} \hspace{3ex}  \overline{B_r(x_1)}\cap\Sigma_2=\{x_0\}.
\end{equation*}
\item  For $i\in\N$ there exists a bijective map $\sigma: \N\rightarrow\N$ with $\sigma(1)=i$ such that for
 every $h\in\N$, $h\geq2$, there is $l\in\N$ with $1\leq l\leq h-1$ and $\Sigma_1\cap\partial\Omega_{\sigma(h)}\cap\partial\Omega_{\sigma(l)}\neq\emptyset$.
\end{itemize}
\end{itemize}

\begin{Prop}\label{lemma1}\textbf{(Theorem 2.2 (Strong Maximum Principle), \cite{1})}\\
Let $u\in C^2(\Omega)\cap C(\overline{\Omega})$ be such that $Lu\geq0$ with $c(x)\leq0$ in $\Omega$ and assume that conditions $(E_\xi)$ and $(\Sigma)$ hold. Then the nonnegative maximum of $u$ in $\overline{\Omega}$ can be attained only on $\partial\Omega$, unless $u$ is constant.
\end{Prop}

\begin{Rem}\label{rem1}
The hypothesis $\displaystyle\sup_{\overline{\Omega}}u=M\geq0$ can be dropped if $c(x)\equiv0$ on $\overline{\Omega}$ (see \cite[Remark 2.11]{1}).
Note that the Grushin operator $\Delta_\gamma$ satisfies $(E_\xi)$ and $(\Sigma)$ (\cite[Remark 4.1 and Lemma 4.2]{1}). Therefore, the strong maximum principle holds for the linear degenerate elliptic operator $\Delta_\gamma$ on a bounded domain $\Omega\subset\R^{N+l}$.
\end{Rem}

\begin{Rem}\label{rem11}
The strong maximum principle also applies to weak solutions in $H^{1,2}_\gamma(\Omega)$ due to the fact that the Grushin operator satisfies mean value formulas in \cite{value}. The mean value theorem does not require the solution to be $C^2$.
\end{Rem}

Our main results can be summarized as follows. We first obtain a statement on the symmetry of solutions.

\begin{Thm}\label{th1.1}
Let $u(z)\in H^{1,2}_\gamma(\R^{N+l})\cap C^0(\R^{N+l})$ be a solution to Equation \eqref{0} with $p>1$. Then $u(z)=u(x,y)$ is radially symmetric with respect to the variable
$y$ and has exponential decay at infinity.
\end{Thm}

In addition we obtain a Liouville type theorem in Euclidean half space $\R^{N+l}_+$.
The corresponding Liouville type theorem in Euclidean space $\R^{N+l}$ is a direct corollary of \cite[Theorem 1.1]{17} by taking $f(t)=t^p$, $1<p<\frac{N_\gamma+2}{N_\gamma-2}$, cf. Lemma \ref{th1.2} below.

\begin{Thm}\label{th1.3}
Let $u(z)\in H^{1,2}_\gamma(\R^{N+l}_+)\cap C^0(\overline{\R^{N+l}_+})$ be a solution to Equation \eqref{004.3} with {$1+\frac{4}{N_\gamma}\leq p<\frac{N_\gamma+2}{N_\gamma-2}$}. If $u$ is bounded in $\R^{N+l}_+$, then $u\equiv0$.
\end{Thm}

Finally, we obtain a priori estimates {and the existence} for positive solutions to \eqref{002.1} {under some appropriate assumptions}.

\begin{Thm}\label{th1.4}
Let $\Omega\subset\R^{N+l}$ be a bounded domain with $C^1$ boundary $\partial\Omega$ containing the origin, i.e. $\Omega\cap\{x=0\}$ is not empty, and $\gamma=m \in\N$. { Let $u\in H^{1,2}_\gamma(\Omega)\cap C^\theta(\overline{\Omega})$ be a positive solution to \eqref{002.1}, where $0<\theta\leq1$}. Moreover, suppose that $f(z,t)$ is continuous and for {$p\in[1+\frac{4}{N_\gamma},\frac{N_\gamma+2}{N_\gamma-2})$} uniformly in $z\in \overline{\Omega}$ it satisfies
\begin{equation}\label{002.2}
 \lim_{t\rightarrow+\infty}\frac{f(z,t)}{t^p}=h(z),
\end{equation}
where $h(z)\in C^\theta(\overline{\Omega})$ is strictly positive in $\overline{\Omega}$.
Then there exists $C>0$, such that for any solution $u$ of \eqref{002.1}
\begin{equation*}
 \sup_{z\in\overline{\Omega}}u(z)\leq C.
 \end{equation*}
Here the constant $C$ only depends on $p$, $\Omega$, the behaviour of $f$ in the limit arising in \eqref{002.2} and is independent of the choice of $u$.
\end{Thm}

{
\begin{Thm}\label{th1.5}
Under the assumptions of Theorem \ref{th1.4}, if further $f$ satisfies:
\begin{itemize}
\item[(i)] $f(z,u)\in C(\overline{\Omega}\times\R^+,\R^+)$;
\item[(ii)]  $\displaystyle\lim_{t\rightarrow+\infty}\frac{f(z,t)}{t}>\lambda_1$ uniformly in $z\in\overline{\Omega}$, where $\lambda_1$ is the first eigenvalue of $(-\Delta_\gamma)$ acting on $H^{1,2}_\gamma(\Omega)$ with Dirichlet boundary condition;
\item[(iii)] $\displaystyle\lim_{t\rightarrow0^+}\frac{f(z,t)}{t}<\lambda_1$ uniformly in $z\in\overline{\Omega}$,
\end{itemize}
then Equation \eqref{002.1} has at least one positive solution $u\in H^{1,2}_\gamma(\Omega)\cap C^\theta(\overline{\Omega})$.
%

\end{Thm}
}

Grushin operators $\Delta_{\gamma}$ with $\gamma \in \mathbb{N}$ form classical examples in the family of all H\"ormander operators. This is one reason why we study the nonlinear degenerate elliptic Equations \eqref{002.1}. The main achievements of the paper can be summarized as follows.
We prove radial symmetry with respect to the variable $y$  of solutions to the Equation \eqref{0} and obtain a nonexistence result for positive solutions to the Equation \eqref{004.3}
{in Euclidean half space.} Furthermore, we appropriately scale the $x$ and $y$ variables, respectively, to deal with the asymmetry of the Grushin operators when defining a blow up sequence in order to establish a priori bounds for positive solutions to \eqref{002.1}.
\vspace{1ex}

The paper is organized as follows: in Section 2, we prove Theorem \ref{th1.1} and as a result we obtain the radial symmetry and the decay rate at infinity of solutions to the Equation \eqref{0} by applying the moving plane method combined with appropriate integral inequalities. In Section 3, we prove Theorem \ref{th1.3} to obtain nonexistence results for solutions to the Equations \eqref{004.3} in $\R^{N+l}_+$. Again our arguments are based on the moving plane method and some integral inequalities. In Section 4, we prove Theorem \ref{th1.4} to obtain a priori estimates for positive solutions to the Equation \eqref{002.1} by employing blow up analysis. In Section 5, we prove Theorem \ref{th1.5} to obtain the existence for positive solutions to the Equation \eqref{002.1} under appropriate conditions by applying topological degree methods.

\section{Symmetry result}

In this section, we prove Theorem \ref{th1.1} and show the radial symmetry and the decay rate at infinity of solutions to the Equation \eqref{0}.

To prepare, we state two results that are essential in the proof of the theorem. Since the domain we consider could be unbounded, we use Lemma \ref{unbounded} to substitute the maximum principles that are known in the case of bounded domains, e.g. see \cite{22}. Lemma \ref{unbounded1} is an important ingredient for the moving plane method.

\begin{Lem}\label{unbounded}
Assume that $\Omega\subset\R^{N+l}$ is an unbounded domain contained in ${\widetilde{B}_{R_0}}(0)^c$ for some $R_0>0$, and $\omega(z)$ is a solution {in $ H^{1,2}_\gamma(\Omega)$} of
\begin{equation}\label{un1}
\begin{cases}
-\Delta_\gamma\omega(z) + c(z)\omega(z)\leq0,\;\;&~ \mbox{in} ~\Omega\\[2mm]
\omega(z)\leq0, ~ &~ \mbox{on} ~\partial\Omega.
\end{cases}
\end{equation}
Suppose there exists $0\leq b<N_\gamma-2$, such that
\begin{equation}\label{un2}
  c(z)>-\frac{b(N_\gamma-2-b)|x|^{2\gamma}}{(1+\gamma)^2d^{2+2\gamma}(z,0)},\;\;~\mbox{if}~d(z,0)\geq R_0,
\end{equation}
and
\begin{equation}\label{un3}
  \lim_{d(z,0)\rightarrow\infty}\omega(z)d^b(z,0)=0.
\end{equation}
Then we have
\begin{equation}\label{un4}
  \omega(z)\leq0,\;\;~\mbox{in}~\Omega.
\end{equation}

\end{Lem}

\begin{proof}

\vskip 0.2cm

\textbf{We first consider the case $0<b<N_\gamma-2$}.

\vskip 0.2cm
Let $\overline{\omega}(z)=\frac{\omega(z)}{\phi(z)}$, where $\phi(z)>0$. Then $\omega(z)=\overline{\omega}(z)\phi(z)$ and we derive
\begin{equation}\label{un5}
  \Delta_\gamma\omega(z)=\phi(z)\Delta_\gamma\overline{\omega}(z)+\overline{\omega}(z)\Delta_\gamma\phi(z)
  +2\nabla_\gamma\overline{\omega}(z)\cdot\nabla_\gamma\phi(z).
\end{equation}
Since $\phi(z)>0$, Equation \eqref{un1} transforms into
\begin{equation}\label{un6}
\begin{cases}
-\Delta_\gamma\overline{\omega}(z) - \frac{2\nabla_\gamma\overline{\omega}(z)\cdot\nabla_\gamma\phi(z)}{\phi(z)} + \left(c(z)-\frac{\Delta_\gamma\phi(z)}{\phi(z)}\right)\overline{\omega}(z)\leq0,\;\;&~ \mbox{in} ~\Omega\\[2mm]
\overline{\omega}(z)\leq0, ~ &~ \mbox{on} ~\partial\Omega.
\end{cases}
\end{equation}
Take
\begin{equation}\label{un7}
  \phi(z)=d^{-b}(z,0),
\end{equation}
then we obtain
\begin{equation}\label{un8}
  \Delta_\gamma\phi(z)=b(b-N_\gamma+2)\frac{1}{(1+\gamma)^2}|x|^{2\gamma}d^{-b-2-2\gamma}(z,0),
\end{equation}
and
\begin{equation}\label{un9}
  -\frac{\Delta_\gamma\phi(z)}{\phi(z)}=\frac{b(N_\gamma-2-b)|x|^{2\gamma}}{(1+\gamma)^2d^{2+2\gamma}(z,0)}.
\end{equation}
By condition \eqref{un2}, one can see that
\begin{equation}\label{un10}
  c(z)-\frac{\Delta_\gamma\phi(z)}{\phi(z)}>0,\;\;~\mbox{if}~d(z,0)\geq R_0.
\end{equation}

By contradiction, suppose that \eqref{un4} is invalid, then $\omega(z)$ is positive somewhere in $\Omega$, and likewise $\overline{\omega}(z)$. According to \eqref{un3}, i.e.
\begin{equation}\label{un11}
  \lim_{d(z,0)\rightarrow\infty}\overline{\omega}(z)=0,
\end{equation}
and $\overline{\omega}(z)\leq0$ on $\partial\Omega$, we know that there exists $z^0\in\Omega$ such that
\begin{equation}\label{un12}
  \overline{\omega}(z^0)=\max_\Omega\overline{\omega}(z)>0.
\end{equation}
We have
\begin{equation}\label{un13}
\begin{aligned}
  &-\Delta_\gamma\overline{\omega}(z^0) - \frac{2\nabla_\gamma\overline{\omega}(z^0)\cdot\nabla_\gamma\phi(z^0)}{\phi(z^0)} + \left(c(z^0)-\frac{\Delta_\gamma\phi(z^0)}{\phi(z^0)}\right)\overline{\omega}(z^0)
\\=&-\Delta_\gamma\overline{\omega}(z^0)
+\left(c(z^0)-\frac{\Delta_\gamma\phi(z^0)}{\phi(z^0)}\right)\overline{\omega}(z^0)
>0.
\end{aligned}
\end{equation}
This contradicts \eqref{un6} and therefore, \eqref{un4} must be valid.

\vskip 0.2cm

\textbf{Next we consider the case $b=0$}.

\vskip 0.2cm

By contradiction, suppose that \eqref{un4} is invalid. Then $\omega(z)$ is positive somewhere in $\Omega$.
According to \eqref{un3}, i.e.
\begin{equation}\label{mns11}
  \lim_{d(z,0)\rightarrow\infty}\omega(z)=0,
\end{equation}
and $\omega(z)\leq0$ on $\partial\Omega$, we know that there exists a $z^0\in\Omega$ such that
\begin{equation}\label{mns12}
  \omega(z^0)=\max_\Omega\omega(z)>0.
\end{equation}
Note that condition \eqref{un2} now states that
\begin{equation}\label{mns13}
  c(z)>0,\;\;~\mbox{if}~d(z,0)\geq R_0.
\end{equation}
At this point, we have
\begin{equation}\label{mns14}
  -\Delta_\gamma\omega(z^0) + c(z^0)\omega(z^0)>0,
\end{equation}
but this contradicts \eqref{un1}. Therefore, \eqref{un4} must be valid.

\end{proof}

\begin{Lem}\label{unbounded1}
Assume that $\Omega$ is a domain in $\R^{N+l}$ which is allowed to be unbounded. Let $\omega(z)$ be a solution in $ H^{1,2}_\gamma(\Omega)$ to
\begin{equation}\label{un14}
\begin{cases}
-\Delta_\gamma\omega(z) + \alpha\omega(z)\leq0,\;\;&~ \mbox{in} ~\Omega\\[2mm]
\omega(z)\leq0, ~ &~ \mbox{on} ~\partial\Omega.
\end{cases}
\end{equation}
Suppose $\alpha\geq0$ is constant and $\omega(z)\leq0$ in $\Omega$. If $\omega(z)\not\equiv0$ in $\Omega$, then we have
\begin{equation}\label{un15}
  \omega(z)<0,\;\;~\mbox{in}~\Omega.
\end{equation}

\end{Lem}

\begin{proof}
We prove the statement by contradiction. Assume that there exists $z^0\in\Omega$, such that $\omega(z^0)=0$. For any given $r_0>0$, such that $\widetilde{B}_{r_0}(z^0)\subset\Omega$, we apply the strong maximum principle in $\widetilde{B}_{r_0}(z^0)$. It follows that the nonnegative maximum of $\omega(z)$ in $\overline{\widetilde{B}_{r_0}}(z^0)$ can be attained only on $\partial\widetilde{B}_{r_0}(z^0)$, unless $\omega(z)$ is constant. However, $\omega(z)$ also attains a non-negative maximum value at $z^0\in\widetilde{B}_{r_0}(z^0)$.
Hence $\omega$ is constant in $\widetilde{B}_{r_0}(z^0)$ with value $\omega(z)\equiv\omega(z^0)=0$.

Take $z_1\in\widetilde{B}_{r_0}(z^0)\setminus\{z^0\}$, then $\omega(z_1)=0$. For any given $r_1>0$, such that $\widetilde{B}_{r_1}(z_1)\subset\Omega$, we use the strong maximum principle as above and obtain $\omega(z)\equiv0$ in $\widetilde{B}_{r_1}(z_1)$. For any $z\in\Omega$, we can use a finite number of balls to connect $z^0$ and $z$. By the method described above, after a finite number of steps, we can obtain $\omega(z)\equiv0$. Due to the arbitrariness of $z$, we know that $\omega(z)\equiv0$ in $\Omega$. This contradicts $\omega(z)\not\equiv0$ in $\Omega$.

\end{proof}
\begin{proof}[\textbf{Proof of Theorem \ref{th1.1}:}]

We will denote by $z=(x,y)=(x,y_1,\widetilde{y})$ points of $\R^{N+l}$, where $x=(x_1\cdots,x_N)$, $\widetilde{y}= (y_2,\cdots,  y_l)$. For any $\lambda\in \mathbb{R}$, we define
\begin{equation}\label{01}
  T_\lambda=\{z\in\R^{N+l}|\;y_1=\lambda\},
\end{equation}

\begin{equation}\label{02}
  \Sigma_\lambda=\{z\in\R^{N+l}|\;y_1>\lambda\},
\end{equation}

\begin{equation}\label{03}
  \widehat{\Sigma}_\lambda=\{~\mbox{reflection of}~\Sigma_\lambda ~\mbox{with respect to the hyperplane}~T_\lambda \},
\end{equation}

\begin{equation}\label{04}
  z^\lambda=(x,2\lambda-y_1,\widetilde{y})\in\widehat{\Sigma}_\lambda,\;\;\;\;~\mbox{for any point}~z=(x,y_1,\widetilde{y})\in\Sigma_\lambda,
\end{equation}
then
\begin{equation}\label{05}
  T_\lambda=\partial\Sigma_\lambda.
\end{equation}
First we consider $\lambda\geq 0$. Given a solution $u(z)$ to (\ref{0}) we put
\begin{equation}\label{06}
  u_\lambda(z):=u(z^\lambda),
\end{equation}
and define a new function
\begin{equation}\label{07}
  \omega_\lambda(z):=u(z)-u_\lambda(z),\;\;\;\;~\mbox{in}~\Sigma_\lambda.
\end{equation}
Note that $\omega_\lambda(z)$ satisfies the equation
\begin{equation}\label{08}
  \begin{cases}
-\Delta_\gamma\omega_\lambda(z) + \omega_\lambda(z)=u^p(z)-u_\lambda^p(z),\;\;&~ \mbox{in} ~\Sigma_\lambda\\[2mm]
\omega_\lambda(z)=0, ~ &~ \mbox{on} ~\partial\Sigma_\lambda.
\end{cases}
\end{equation}

We divide the proof into the following five steps.

\vskip 0.2cm

\noindent\textbf{Step 1: Solutions to \eqref{0} are exponentially decaying at infinity.}

\vskip 0.2cm

For some given constants $a,c,R>0$, we define
\begin{equation}\label{1}
  \Psi(z)=ce^{-a(d(z,0)-R)},
\end{equation}
where $d(z,0)$ was defined in (\ref{006.6}).
We verify that
\begin{equation}\label{2'}
  \begin{cases}
  -\Delta_\gamma\Psi(z)+a^2\Psi(z)\geq0, ~\mbox{for}~d(z,0)>R,
\\\lim\limits_{d(z,0)\to \infty}\Psi(z)=0.
\end{cases}
\end{equation}

Indeed, for $i=1,\cdots,N$, $j=1,\cdots,l$, we have
\begin{align}
\frac{\partial d(z,0)}{\partial x_i}
&=\frac{1}{(1+\gamma)^2}d^{-1-2\gamma}(z,0)|x|^{2\gamma}x_i,\label{3}\\
 \frac{\partial d(z,0)}{\partial y_j}
 &=\frac{1}{1+\gamma}d^{-1-2\gamma}(z,0)y_j, \label{4}\\
 \Psi_{x_i}(z)
 &=-\frac{ac}{(1+\gamma)^2}e^{-a(d(z,0)-R)}d^{-1-2\gamma}(z,0)|x|^{2\gamma}x_i,\label{5}\\
 \Psi_{y_j}(z)
 &=-\frac{ac}{1+\gamma}e^{-a(d(z,0)-R)}d^{-1-2\gamma}(z,0)y_j,
\end{align}
and
\begin{multline}\label{6}
  \Psi_{x_ix_i}(z)=-\frac{ac}{(1+\gamma)^2}e^{-a(d(z,0)-R)}\Bigg(-\frac{a}{(1+\gamma)^2}d^{-2-4\gamma}(z,0)|x|^{4\gamma}x_i^2\\
-\frac{1+2\gamma}{(1+\gamma)^2}d^{-3-4\gamma}(z,0)|x|^{4\gamma}x_i^2
+2\gamma d^{-1-2\gamma}(z,0)|x|^{2\gamma-2}x_i^2
+d^{-1-2\gamma}(z,0)|x|^{2\gamma}\Bigg),
\end{multline}

\begin{multline}\label{6'}
  \Psi_{y_jy_j}(z)=-\frac{ac}{1+\gamma}e^{-a(d(z,0)-R)}\Bigg(-\frac{a}{1+\gamma}d^{-2-4\gamma}(z,0)y_j^2\\
  -\frac{1+2\gamma}{1+\gamma}d^{-3-4\gamma}(z,0)y_j^2
  +d^{-1-2\gamma}(z,0)\Bigg).
\end{multline}
Thus, it holds
\begin{multline}\label{7}
  \Delta_x\Psi(z)=-\frac{ac}{(1+\gamma)^2}e^{-a(d(z,0)-R)}\Bigg(-\frac{a}{(1+\gamma)^2}d^{-2-4\gamma}(z,0)|x|^{4\gamma+2}\\
-\frac{1+2\gamma}{(1+\gamma)^2}d^{-3-4\gamma}(z,0)|x|^{4\gamma+2}
+(2\gamma+N)d^{-1-2\gamma}(z,0)|x|^{2\gamma}\Bigg),
\end{multline}
and
\begin{multline}\label{7'}
  \Delta_y\Psi(z)=-\frac{ac}{1+\gamma}e^{-a(d(z,0)-R)}\Bigg(-\frac{a}{1+\gamma}d^{-2-4\gamma}(z,0)|y|^2\\
  -\frac{1+2\gamma}{1+\gamma}d^{-3-4\gamma}(z,0)|y|^2
  +ld^{-1-2\gamma}(z,0)\Bigg).
\end{multline}
From $\Delta_{\gamma}= \Delta_x+|x|^{2\gamma} \Delta_y$ and \eqref{7} and \eqref{7'} we obtain
\begin{equation}\label{8}
  \Delta_\gamma\Psi(z)=a^2\Psi(z)\frac{1}{(1+\gamma)^2}d^{-2\gamma}(z,0)|x|^{2\gamma}-\Psi(z)\frac{N_\gamma-1}{(1+\gamma)^2}ad^{-1-2\gamma}(z,0)|x|^{2\gamma},
\end{equation}
where $N_\gamma=N+(1+\gamma)l$. So we have
\begin{equation}\label{008.1}
  -\Delta_\gamma\Psi(z)+a^2\Psi(z)\frac{1}{(1+\gamma)^2}d^{-2\gamma}(z,0)|x|^{2\gamma}\geq0.
\end{equation}
It is easy to see that
\begin{equation}\label{008.2}
  \frac{1}{(1+\gamma)^2}\left(\frac{|x|}{d(z,0)}\right)^{2\gamma}\leq1.
\end{equation}
In fact, according to the definition of $d(z,0)$,
\begin{equation}\label{009.1}
  \frac{1}{(1+\gamma)^2}|x|^{2+2\gamma}+|y|^2=d^{2+2\gamma}(z,0),
\end{equation}
\begin{equation}\label{009.2}
  \frac{1}{(1+\gamma)^2}|x|^{2+2\gamma}\leq d^{2+2\gamma}(z,0),
\end{equation}
\begin{equation}\label{009.3}
  \frac{1}{(1+\gamma)^{\frac{1}{1+\gamma}}}|x|\leq d(z,0),
\end{equation}
\begin{equation}\label{009.4}
\frac{1}{(1+\gamma)^{\frac{2\gamma}{1+\gamma}}}\left(\frac{|x|}{d(z,0)}\right)^{2\gamma}\leq1,
\end{equation}
showing \eqref{008.2}. Finally \eqref{2'} follows by combining  \eqref{008.2} and \eqref{008.1}. Let $u(z)$ be a solution to \eqref{0}.
With $R>0$ sufficiently large and a suitable constant $a>0$ we will prove:
\begin{equation}\label{9}
  u(z)\leq\Psi(z), ~\mbox{for}~d(z,0)>R.
\end{equation}

We will derive a contradiction by assuming that there exists some $z^0\in\R^{N+l}$ satisfying $d(z^0,0)>R$, such that $u(z^0)>\Psi(z^0)$.
Set
\begin{equation}\label{10}
  \Omega=\{z\in\R^{N+l}|\;d(z,0)>R,\;u(z)>\Psi(z)\}\neq\emptyset,
\end{equation}
and observe that $(u(z)-\Psi(z))|_{\partial\Omega}=0$. By taking the difference of the following relations
\begin{equation}\label{11}
    \begin{cases}
-\Delta_\gamma\Psi(z)+a^2\Psi(z)\geq0,\\[2mm]
-\Delta_\gamma u(z)+ u(z)-u^p(z)=0,
\end{cases}
\end{equation}
we have
\begin{equation}\label{12}
  -\Delta_\gamma(u(z)-\Psi(z))+(1-u^{p-1}(z))u(z)-a^2\Psi(z)\leq0.
\end{equation}
Recall that $u(z) \rightarrow 0$ as $d(z,0) \rightarrow \infty$. Hence, for $R>0$ suffficiently large,  we can choose $a >0$ such that:
\begin{equation}\label{13}
  1-u^{p-1}(z)>a^2, ~\mbox{for}~d(z,0)>R.
\end{equation}
Then \eqref{12} transforms into
\begin{equation}\label{14}
  \begin{cases}
  -\Delta_\gamma(u(z)-\Psi(z))+a^2(u(z)-\Psi(z))\leq0,\;\;~ &\mbox{in} ~\Omega\\[2mm]
u(z)-\Psi(z)=0,\;\;~ &\mbox{on} ~\partial\Omega.
\end{cases}
\end{equation}
Take $\omega(z)=u(z)-\Psi(z)$, since $u(z)\rightarrow0$ as $d(z,0)\rightarrow\infty$ and $\Psi(z)=ce^{-a(d(z,0)-R)}$, the function $\omega(z)$ clearly  satisfies \eqref{un3} with $b=0$. According to Lemma \ref{unbounded}, we conclude that $u(z)-\Psi(z)\leq0$ in $\Omega$, i.e. $u(z)\leq\Psi(z)$ in $\Omega$. This contradicts the definition of $\Omega$. So we have proved \eqref{9}.

\vskip 0.2cm

\noindent\textbf{Step 2: We show $\omega_\lambda(z)\leq0$ in $\Sigma_\lambda$ for $\lambda >0$ sufficiently large.}

\vskip 0.2cm

Let $\varepsilon >0$ and construct the cutoff function $\eta_\varepsilon\in C_0^\infty(\Sigma_\lambda)$ satisfying
\begin{equation}\label{008.6}
\eta_\varepsilon(z) = \left\{\begin{array}{ll}
1, &~\mbox{if}~d(z,0)\leq\frac{1}{\varepsilon} ,\\[2mm]
0, &~\mbox{if}~d(z,0)\geq\frac{2}{\varepsilon},
\end{array}\right.
\end{equation}
and with $C>0$ (independent of $\varepsilon$):
\begin{equation}\label{008.7}
  |\nabla_\gamma\eta_\varepsilon(z)|\leq C\varepsilon \hspace{3ex} ~\mbox{\it if} \hspace{3ex} \frac{1}{\varepsilon}<d(z,0)<\frac{2}{\varepsilon}.
\end{equation}
Then there is pointwise convergence
\begin{equation}\label{oll}
  \eta_\varepsilon\rightarrow1,\;\;\;\;~\mbox{as}~\varepsilon\rightarrow0.
\end{equation}

We test the Equation \eqref{08} with the function $\phi(z):=\omega_\lambda^{+}(z)\eta_\varepsilon^2(z)$, where with the notation in (\ref{07}) we write
$\omega_{\lambda}^+$ for the positive part of $\omega_{\lambda}$:
\begin{equation}\label{17}
\omega_\lambda^{+}(z):=\left\{\begin{array}{ll}
\omega_\lambda(z), & \omega_\lambda(z)>0 ,\\
0, & \omega_\lambda(z)\leq0.
\end{array}\right.
\end{equation}
Note that $\phi(z)=0$ on $\partial\Sigma_\lambda$. Introducing the notation $\Sigma_\lambda^{+}:=\{z\in\Sigma_\lambda|\;\omega_\lambda(z)>0\}$
we have
\begin{equation}\label{18}
  \int_{\Sigma_\lambda}(-\Delta_\gamma\omega_\lambda(z) + \omega_\lambda(z))\phi(z)dz=\int_{\Sigma_\lambda}(u^p(z)-u_\lambda^p(z))\phi(z)dz.
\end{equation}

First we calculate the LHS of \eqref{18} by using Green's first formula:
{\allowdisplaybreaks
\begin{equation}\label{19}
\begin{aligned}
  &\int_{\Sigma_\lambda}(-\Delta_\gamma\omega_\lambda(z)+\omega_\lambda(z))\phi(z)dz
 \\=&\int_{\Sigma_\lambda}-\Delta_\gamma\omega_\lambda(z)(\omega_\lambda^{+}(z)\eta_\varepsilon^2(z))dz
 +\int_{\Sigma_\lambda}\omega_\lambda(z)\omega_\lambda^{+}(z)\eta_\varepsilon^2(z)dz
  \\=&\int_{\Sigma_\lambda^+}-\Delta_\gamma\omega_\lambda^+(z)(\omega_\lambda^{+}(z)\eta_\varepsilon^2(z))dz
 +\int_{\Sigma_\lambda^+}(\omega_\lambda^{+}(z)\eta_\varepsilon(z))^2dz
 \\=&\int_{\Sigma_\lambda^+}\nabla_\gamma\omega_\lambda^+(z)\cdot
 \nabla_\gamma(\omega_\lambda^{+}(z)\eta_\varepsilon^2(z))dz
 +\int_{\Sigma_\lambda^+}(\omega_\lambda^{+}(z)\eta_\varepsilon(z))^2dz
 \\=&\int_{\Sigma^{+}_\lambda}\eta_\varepsilon^2(z)|\nabla_\gamma\omega_\lambda^+(z)|^2
 +2\omega_\lambda^+(z)\eta_\varepsilon(z)\nabla_\gamma\eta_\varepsilon(z)\cdot\nabla_\gamma\omega_\lambda^+(z)dz
 +\int_{\Sigma^{+}_\lambda}(\omega_\lambda^+(z)\eta_\varepsilon(z))^2dz
 \\=&\int_{\Sigma_\lambda^+}|\nabla_\gamma(\omega^{+}_\lambda(z)\eta_\varepsilon(z))|^2dz
 -\int_{\Sigma_\lambda^+}(\omega_\lambda^{+}(z))^2|\nabla_\gamma\eta_\varepsilon(z)|^2dz
 +\int_{\Sigma_\lambda^+}(\omega_\lambda^{+}(z)\eta_\varepsilon(z))^2dz
 \\=&\|\omega_\lambda^+(z)\eta_\varepsilon(z)\|^2_{H^{1,2}_\gamma(\Sigma_\lambda^+)}
 -\int_{\Sigma_\lambda}(\omega_\lambda^{+}(z))^2|\nabla_\gamma\eta_\varepsilon(z)|^2dz.
\end{aligned}
\end{equation}
}

Then we calculate the RHS of \eqref{18} by using the mean value theorem and Step 1. Note that $\omega_\lambda(z)>0$ in $\Sigma_\lambda^+$,
i.e. $u(z)>u_\lambda(z)$. So in $\Sigma_\lambda^+$,
{\allowdisplaybreaks
\begin{equation}\label{oll1}
\begin{aligned}
  (u^p(z)-u_\lambda^p(z))\omega_\lambda^{+}(z)\eta_\varepsilon^2(z)
\leq pu^{p-1}(z)\omega_\lambda(z)\omega_\lambda^{+}(z)\eta_\varepsilon^2(z)
= pu^{p-1}(z)(\omega_\lambda^+(z)\eta_\varepsilon(z))^2.
\end{aligned}
\end{equation}
Then, with $2_\gamma^\ast=\frac{2N_\gamma}{N_\gamma-2}$, suitable constants $C,a>0$ and $\lambda >0$ sufficiently large such that $
u \leq \Psi$ on $\Sigma_{\lambda}^+$, see (\ref{9}):
\begin{equation}\label{20}
\begin{aligned}
  &\int_{\Sigma_\lambda}(u^p(z)-u_\lambda^p(z))\omega_\lambda^{+}(z)\eta_\varepsilon^2(z)dz
\\=&\int_{\Sigma_\lambda^+}(u^p(z)-u_\lambda^p(z))\omega_\lambda^{+}(z)\eta_\varepsilon^2(z)dz
\\\leq&\int_{\Sigma_\lambda^+}pu^{p-1}(z)(\omega_\lambda^+(z)\eta_\varepsilon(z))^2dz
\\\leq&p\left(\int_{\Sigma_\lambda^+}(u(z))^{(p-1)\frac{N_\gamma}{2}}dz\right)^{\frac{2}{N_\gamma}}
  \left(\int_{\Sigma_\lambda^+}\left(\omega_\lambda^+(z)\eta_\varepsilon(z)
  \right)^{\frac{2N_\gamma}{N_\gamma-2}}dz\right)^{\frac{N_\gamma-2}{N_\gamma}}
\\\leq&C\left(\int_{\Sigma_\lambda^+}e^{-a(d(z,0)-R)(p-1)\frac{N_\gamma}{2}}dz\right)^{\frac{2}{N_\gamma}}
  \left(\int_{\Sigma_\lambda^+}\left(\omega_\lambda^+(z)\eta_\varepsilon(z)
  \right)^{\frac{2N_\gamma}{N_\gamma-2}}dz\right)^{\frac{N_\gamma-2}{N_\gamma}}
\\=&C\left(\int_{\Sigma_\lambda^+}e^{-a(d(z,0)-R)(p-1)\frac{N_\gamma}{2}}dz\right)^\frac{2}{N_\gamma}
\|\omega_\lambda^{+}(z)\eta_\varepsilon(z)\|^2_{L^{2_\gamma^\ast}(\Sigma_\lambda^+)}.
\end{aligned}
\end{equation}}
According to \eqref{18} we have
\begin{equation}\label{21}
\begin{aligned}
  &\|\omega_\lambda^{+}(z)\eta_\varepsilon(z)\|^2_{H^{1,2}_\gamma(\Sigma_\lambda^+)}
  \\\leq&C\left(\int_{\Sigma_\lambda^+}e^{-a(d(z,0)-R)(p-1)\frac{N_\gamma}{2}}dz\right)^\frac{2}{N_\gamma}
\|\omega_\lambda^{+}(z)\eta_\varepsilon(z)\|^2_{L^{2_\gamma^\ast}(\Sigma_\lambda^+)}
+\int_{\Sigma_\lambda}(\omega_\lambda^+(z))^2|\nabla_\gamma\eta_\varepsilon(z)|^2dz
\\=:&\uppercase\expandafter{\romannumeral1}+\uppercase\expandafter{\romannumeral2}.
\end{aligned}
\end{equation}
Since $u$ is nonnegative we conclude that $\omega_\lambda^{+}(z)\leq u(z)\in L^{2_\gamma^\ast}(\Sigma_\lambda)$. With the notation
$A_{\varepsilon}:=\{z\in\R^{N+l}\; |\;\frac{1}{\varepsilon}<d(z,0)<\frac{2}{\varepsilon}\}$ we have:
\begin{equation}\label{22}
\begin{aligned}
  \uppercase\expandafter{\romannumeral2}
  =&\int_{{\Sigma_\lambda}\cap A_{\varepsilon}}(\omega_\lambda^{+}(z))^2|\nabla_\gamma\eta_\varepsilon(z)|^2dz
  \\\leq&\left(\int_{{\Sigma_\lambda}\cap A_{\varepsilon}}(\omega_\lambda^+(z))^\frac{2N_\gamma}{N_\gamma-2}dz\right)^\frac{N_\gamma-2}{N_\gamma}
  \left(\int_{{\Sigma_\lambda}\cap A_{\varepsilon}}(|\nabla_\gamma\eta_\varepsilon(z)|)^{N_\gamma}dz\right)^\frac{2}{N_\gamma}
  \\\leq&\left(\int_{A_{\varepsilon}}(\omega_\lambda^+(z))^{2_\gamma^\ast}dz\right)^\frac{2}{2_\gamma^\ast}
  \left(\int_{A_{\varepsilon}}(|\nabla_\gamma\eta_\varepsilon(z)|)^{N_\gamma}dz\right)^\frac{2}{N_\gamma}.
\end{aligned}
\end{equation}
According to (\ref{008.7}) we have the estimate
\begin{equation}\label{23}
  \left(\int_{A_{\varepsilon}}(|\nabla_\gamma\eta_\varepsilon(z)|)^{N_\gamma}dz\right)^\frac{2}{N_\gamma}
  \leq\left((C\varepsilon)^{N_\gamma}\frac{1}{\varepsilon^{N_\gamma}}\right)^\frac{2}{N_\gamma}
  \leq C,
\end{equation}
and from
\begin{equation}\label{24}
  \left(\int_{A_{\varepsilon}}(\omega_\lambda^+(z))^{2_\gamma^\ast}dz\right)^\frac{2}{2_\gamma^\ast}
  \leq\left(\int_{A_{\varepsilon}}(u(z))^{2_\gamma^\ast}dz\right)^\frac{2}{2_\gamma^\ast}
  \rightarrow0,\;\;\;\;~\mbox{as}~\varepsilon\rightarrow0,
\end{equation}
it follows that
\begin{equation}\label{24}
  \uppercase\expandafter{\romannumeral2}\rightarrow0,\;\;\;\;~\mbox{as}~\varepsilon\rightarrow0.
\end{equation}
Therefore, by taking the limit $\varepsilon\rightarrow0$ in \eqref{21} we obtain
\begin{equation}\label{25}
  \|\omega_\lambda^{+}(z)\|^2_{H^{1,2}_\gamma(\Sigma_\lambda^+)}
  \leq C\left(\int_{\Sigma_\lambda^+}e^{-a(d(z,0)-R)(p-1)\frac{N_\gamma}{2}}dz\right)^\frac{2}{N_\gamma}
\|\omega_\lambda^{+}(z)\|^2_{L^{2_\gamma^\ast}(\Sigma_\lambda^+)}.
\end{equation}
Since the embedding $H^{1,2}_\gamma(\R^{N+l})\hookrightarrow L^{2_\gamma^\ast}(\R^{N+l})$ is continuous (see \cite{4}), \eqref{25} implies (possibly with another $C>0$):
\begin{equation}\label{26}
  \|\omega_\lambda^{+}(z)\|^2_{H^{1,2}_\gamma(\Sigma_\lambda^+)}
  \leq C\left(\int_{\Sigma_\lambda^+}e^{-a(d(z,0)-R)(p-1)\frac{N_\gamma}{2}}dz\right)^\frac{2}{N_\gamma}
\|\omega_\lambda^{+}(z)\|^2_{H^{1,2}_\gamma(\Sigma_\lambda^+)}.
\end{equation}
If the parameter $\lambda$ is taken sufficiently large such that
$C\left(\displaystyle\int_{\Sigma_\lambda^+}e^{-a(d(z,0)-R)(p-1)\frac{N_\gamma}{2}}dz\right)^\frac{2}{N_\gamma}<1$,
then we conclude from the last inequality that $\omega_\lambda^{+}(z)=0$. That is, for sufficiently large value of $\lambda>0$ we have
$\omega_\lambda(z)\leq0$ in $\Sigma_\lambda$.

\vskip 0.2cm

\noindent\textbf{Step 3: The case $\omega_\lambda(z)\equiv0$ in $\Sigma_\lambda$ for any $\lambda\in\R$, can not occur.}

\vskip 0.2cm

Recall that $\omega_\lambda(z)$ satisfies \eqref{08}. By Step 2 we know that $\omega_\lambda(z)\leq0$ in $\Sigma_\lambda$ whenever $\lambda$ is
sufficiently large and in this case $u(z)\leq u_\lambda(z)$. Since $p>1$, we have $u^p(z)\leq u_\lambda^p(z)$. Then \eqref{08} implies
\begin{equation}\label{005.1}
\begin{cases}
  -\Delta_\gamma\omega_\lambda(z)+\omega_\lambda(z)\leq0,\;\;~ &\mbox{in} ~\Sigma_\lambda\\[2mm]
\omega_\lambda(z)=0,\;\;~ &\mbox{on} ~\partial\Sigma_\lambda.
\end{cases}
\end{equation}

If $\omega_\lambda(z)\equiv0$ in $\Sigma_\lambda$, i.e. $u(z)=u_\lambda(z)$ in $\Sigma_\lambda$ for any $\lambda\in\R$ when the plane $T_\lambda$ is moving along the $y_1$-axis from near $+\infty$ to the left. Then $u$ is constant with respect to the variable $y_1$. Since $y_1$ is arbitrary, we may apply the same arguments to the remaining components of $y$. It follows that the function $u$ is constant with respect to $y$. Hence we can write $u(x,y)=u(x)$ for each $z=(x,y)\in\R^{N+l}$. According to step 1 the function $u$ is exponentially vanishing as $d(z,0) \rightarrow \infty$, where $z=(x,y)$. Hence the solution $u(x,y)= u(x)= \lim_{y \rightarrow \infty} u(x,y)=0$ must have the constant value zero, $u\equiv 0$. Hence we obtain a contradiction to the assumption $u>0$.

\vskip 0.2cm
\noindent\textbf{Step 4: If $\omega_\lambda(z)\not\equiv0$ in $\Sigma_\lambda$ for some $\lambda$.
Then $\Lambda=\inf\{\lambda>0\; |\;\omega_\lambda(z)\leq0~\mbox{in}~\Sigma_\lambda\}=0$.}

\vskip 0.2cm

Assume that $\omega_\lambda(z)\not\equiv0$ in $\Sigma_\lambda$ for some $\lambda$ when the plane $T_\lambda$ is moving along the $y_1$-axis from near $+\infty$ to the left. By Lemma \ref{unbounded1} and \eqref{005.1} we conclude that $\omega_\lambda(z)<0$ in $\Sigma_\lambda$.
Suppose by contradiction that $\Lambda>0$. By continuity we have $\omega_\Lambda<0$ in $\Sigma_\Lambda$. As in Step 2 we put $\Sigma_\lambda^{+}=\{z\in\Sigma_\lambda|\;\omega_\lambda(z)>0\}$.
Since $\Sigma_\Lambda^{+}= \emptyset$, we have $|\Sigma_\lambda^{+}|\rightarrow0$ as $\lambda\rightarrow\Lambda^{-}$. Multiplying both sides of \eqref{08} by $\phi(z)=\omega_\lambda^{+}(z)\eta_\varepsilon^2(z)$ for $\lambda<\Lambda$ and integrating over $\Sigma_{\lambda}$ we obtain \eqref{18}.
Here we use the notations of Step 2 and by the estimates there we recover (\ref{26}), i.e.
\begin{equation}\label{31}
  \|\omega_\lambda^{+}(z)\|^2_{H^{1,2}_\gamma(\Sigma_\lambda^+)}
  \leq C\left(\int_{\Sigma^{+}_\lambda}e^{-a(d(z,0)-R)(p-1)\frac{N_\gamma}{2}}dz\right)^\frac{2}{N_\gamma}
\|\omega_\lambda^{+}(z)\|^2_{H^{1,2}_\gamma(\Sigma_\lambda^+)}.
\end{equation}

Let $\lambda\rightarrow\Lambda^{-}$, such that $C\left(\displaystyle\int_{\Sigma^{+}_\lambda}e^{-a(d(z,0)-R)(p-1)\frac{N_\gamma}{2}}dz\right)^\frac{2}{N_\gamma}<1$.
Then we have $\omega_\lambda^{+}(z)=0$, i.e. $\omega_\lambda(z)\leq0$. This contradicts the definition of $\Lambda$. Thus we have proved $\Lambda=0$
which implies that $u(x,y_1,\widetilde{y})\leq u(x,-y_1,\widetilde{y})$.

\vskip 0.2cm

\noindent\textbf{Step 5: Show that $u(x,y_1,\widetilde{y})= u(x,-y_1,\widetilde{y})$}.

\vskip 0.2cm
We now consider $T_\lambda$ for $\lambda$ near $-\infty$ and perform an analysis similar to Step 2 - Step 4. Then we have $u(x,y_1,\widetilde{y})=u(x,-y_1,\widetilde{y})$.

Note that $y_1$ was chosen arbitrarily and any orthogonal transformation $O\in O(l)$ of the $y$-variable in $u(x,y)$ again gives a solution $u(x,Oy)$ to which the same arguments apply. Hence we conclude that $u(x,y)$, after composition with an arbitrary element in $O(l)$, remains invariant under all coordinate reflections in $y$. As is well-known (a proof is based on the Cartan Dieudonn\'{e} Theorem)
this means that $u$ is radially symmetric with respect to the $y$-variables, i.e. $u(x,y)=u(x,|y|)$. This finishes the proof of the theorem.
\end{proof}

\section{Liouville type theorem in $\R^{N+l}_+$}

In this section, we prove nonexistence results for solutions to the Equation \eqref{004.3} in $\R^{N+l}_+$.
Again we apply the moving plane method and derive some integral inequalities in order to prove Theorem \ref{th1.3}.
We start by recalling a Liouville type theorem for nonlinear elliptic equations involving the Grushin operator by X. Yu in \cite{17}, which is an important ingredient in our proofs.

\begin{Prop}\label{ppap}\textbf{(Theorem 1.1, \cite{17})}\\
Let $f\in C^0(\R,\R)$ be given such that
\begin{itemize}
\item[$(f1)$] $f$ is nondecreasing in $(0,+\infty)$,
\item[$(f2)$] the quotient $g(t)=\frac{f(t)}{t^\frac{Q+2}{Q-2}}$ is nonincreasing in $(0,+\infty)$ and not constant, where $Q$ denotes the homogeneous dimension
$Q=m+(\alpha+1)k$.
\end{itemize}
Let $$u\in E:=\left\{u\in C^0(\R^m\times\R^k)\;\left|\right.\;\displaystyle\int_{\R^m\times\R^k}|\nabla_xu|^2+(\alpha+1)^2|x|^{2\alpha}|\nabla_yu|^2dxdy<
\infty\right\}$$ be a nonnegative weak solution of the following problem
\begin{equation}\label{yuxiaohui}
  -\left(\Delta_xu+(\alpha+1)^2|x|^{2\alpha}\Delta_yu\right)=f(u),\;~\mbox{in}~\R^n=\R^m\times\R^k,
\end{equation}
where $\alpha>0$ and $(x,y)\in\R^m\times\R^k$. Then $u\equiv0$.
\end{Prop}
In particular, by taking $f(t)=t^p$ and making a dilation in $y$, $1<p<\frac{N_\gamma+2}{N_\gamma-2}$, we obtain the following Liouville type theorem in $\R^{N+l}$.
\begin{Lem}\label{th1.2}
Let $u(z)\in H^{1,2}_\gamma(\R^{N+l})\cap C^0(\R^{N+l})$ be a solution to Equation \eqref{32} and assume that {$1<p<\frac{N_\gamma+2}{N_\gamma-2}$}, then $u\equiv0$.
\end{Lem}
\begin{proof}[\textbf{Proof of Theorem \ref{th1.3}:}]
Let $u(z)\in H^{1,2}_\gamma(\R^{N+l}_+)\cap C^0(\overline{\R^{N+l}_+})$ be a bounded solution to Equation \eqref{004.3} and assume that $1+\frac{4}{N_\gamma}\leq p<\frac{N_\gamma+2}{N_\gamma-2}$.
\vspace{1ex}\\
For any $\lambda>0$, we will use the following notations:
\begin{equation}\label{xj1}
  T'_\lambda:=\{z\in\R^{N+l}_+\;|\;y_l=\lambda\},
\end{equation}

\begin{equation}\label{xj2}
  \Sigma'_\lambda:=\{z\in\R^{N+l}_+\;|\;0<y_l<\lambda\},
\end{equation}

\begin{equation}\label{xj3}
  \partial\Sigma'_\lambda:=T'_\lambda\cup\{y_l=0\},
\end{equation}

\begin{equation}\label{xj5}
  z'^\lambda:=(x,\widetilde{y}',2\lambda-y_l),
\end{equation}
where
\begin{align}
 z:&=(x,\widetilde{y}',y_l)\in\Sigma'_\lambda, \label{xj4}\\
 x:&=(x_1,\cdots,x_N),\label{xj6} \\
 \widetilde{y}':&=(y_1,\cdots,y_{l-1}).\label{xj7}
\end{align}
%
%

We set
\begin{equation}\label{xj8}
  u_\lambda(z):=u(z'^\lambda),\;z\in\Sigma'_\lambda,
\end{equation}
and
\begin{equation}\label{xj9}
  \omega_\lambda(z):=u(z)-u_\lambda(z),\;z\in\Sigma'_\lambda.
\end{equation}
Note that $u(z)=u_\lambda(z)$ on $T'_\lambda$, and $u(z)=0<u_\lambda(z)$ on $\{y_l=0\}$. Thus, $\omega_\lambda(z)$ satisfies
\begin{equation}\label{004.14}
\begin{cases}
-\Delta_\gamma \omega_\lambda(z)=u^p(z)-u^p_\lambda(z),&~ \mbox{in} ~\Sigma'_\lambda,\\[2mm]
\omega_\lambda(z)\leq0,  &~ \mbox{on} ~\partial\Sigma'_\lambda.
\end{cases}
\end{equation}
We divide the proof of the theorem into the following three steps.

\vskip 0.2cm

\noindent\textbf{Step 1: Proof of $\omega_\lambda(z)\leq0$ in $\Sigma'_\lambda$ for $\lambda >0$ sufficiently small}.

\vskip 0.2cm

Let $\varepsilon>0$ and construct the cutoff function $\zeta_\varepsilon(z)\in C_0^\infty(\Sigma'_\lambda)$ satisfying
\begin{equation}\label{004.15}
\zeta_\varepsilon(z)=\left\{\begin{array}{ll}
0, &~\mbox{if}~d(z,0)\leq\varepsilon ,\\[2mm]
1, &~\mbox{if}~d(z,0)\geq2\varepsilon,
\end{array}\right.
\end{equation}
and such that
\begin{equation}\label{004.16}
  |\nabla_\gamma\zeta_\varepsilon(z)|\leq \frac{C}{\varepsilon},\hspace{3ex} \mbox{when} \hspace{3ex} \varepsilon<d(z,0)<2\varepsilon.
\end{equation}
We have the pointwise convergence
\begin{equation}\label{xj10}
  \zeta_\varepsilon\rightarrow1, \;~\mbox{as}~\varepsilon\rightarrow0.
\end{equation}
We test Equation \eqref{004.14} with the function $\psi(z)=\omega_\lambda^{+}(z)\zeta_\varepsilon^2(z)$, where
\begin{equation}\label{004.17}
\omega_\lambda^{+}(z):=\left\{\begin{array}{ll}
\omega_\lambda(z), & \omega_\lambda(z)>0 ,\\
0, & \omega_\lambda(z)\leq0.
\end{array}\right.
\end{equation}
Note that $\psi(z)=0$ on $\partial\Sigma'_\lambda$ and set ${\Sigma'_\lambda}^+=\{z\in\Sigma'_\lambda|\;\omega_\lambda(z)>0\}$.
Then, by (\ref{004.14}) we have
\begin{equation}\label{004.18}
  \int_{\Sigma'_\lambda}-\Delta_\gamma\omega_\lambda(z)\psi(z)dz
  =\int_{\Sigma'_\lambda}(u^p(z)-u_\lambda^p(z))\psi(z)dz.
\end{equation}
First, we calculate the LHS of \eqref{004.18} by using Green's first formula:
\begin{equation}\label{004.19}
  \begin{aligned}
   &\int_{\Sigma_\lambda'}-\Delta_\gamma\omega_\lambda(z)\psi(z)dz
\\=&\int_{\Sigma_\lambda'}-\Delta_\gamma\omega_\lambda(z)(\omega_\lambda^{+}(z)\zeta_\varepsilon^2(z))dz
\\=&\int_{{\Sigma'_\lambda}^+}-\Delta_\gamma\omega_\lambda^+(z)(\omega_\lambda^{+}(z)\zeta_\varepsilon^2(z))dz
 \\=&\int_{{\Sigma'_\lambda}^+}\nabla_\gamma\omega_\lambda^+(z)\cdot
 \nabla_\gamma(\omega_\lambda^{+}(z)\zeta_\varepsilon^2(z))dz
 \\=&\int_{{\Sigma'_\lambda}^+}\zeta_\varepsilon^2(z)|\nabla_\gamma\omega_\lambda^+(z)|^2
 +2\omega_\lambda^+(z)\zeta_\varepsilon(z)\nabla_\gamma\zeta_\varepsilon(z)\cdot\nabla_\gamma\omega_\lambda^+(z)dz
\\=&\int_{{\Sigma'_\lambda}^+}|\nabla_\gamma\left(\omega_\lambda^{+}(z)\zeta_\varepsilon(z)\right)|^2dz
-\int_{\Sigma'_\lambda}\left(\omega_\lambda^{+}(z)\right)^2|\nabla_\gamma\zeta_\varepsilon(z)|^2dz.
\end{aligned}
\end{equation}
According to \eqref{004.18}, we have
\begin{equation}\label{004.21}
\begin{aligned}
 \int_{{\Sigma'_\lambda}^+}|\nabla_\gamma\left(\omega_\lambda^{+}(z)\zeta_\varepsilon(z)\right)|^2dz
=&\int_{\Sigma'_\lambda}\left(u^p(z)-u^p_\lambda(z)\right)\omega_\lambda^{+}(z)\zeta_\varepsilon^2(z)dz
\\+&\int_{\Sigma'_\lambda}\left(\omega_\lambda^{+}(z)\right)^2|\nabla_\gamma\zeta_\varepsilon(z)|^2dz
:=\uppercase\expandafter{\romannumeral1}+\uppercase\expandafter{\romannumeral2}.
\end{aligned}
\end{equation}

We estimate $\uppercase\expandafter{\romannumeral2}$ and introduce the notation $B_{\varepsilon}:=\{z\in\R^{N+l}_+\;|\; \varepsilon<d(z,0)<2\varepsilon\}$. Then
\begin{equation}\label{xj15}
\begin{aligned}
  \uppercase\expandafter{\romannumeral2}
  &=\int_{\Sigma'_\lambda\cap B_{\varepsilon}}\left(\omega_\lambda^{+}(z)\right)^2|\nabla_\gamma\zeta_\varepsilon(z)|^2dz
\\&\leq\left(\int_{{\Sigma'_\lambda}\cap B_{\varepsilon}}(\omega_\lambda^+(z))^\frac{2N_\gamma}{N_\gamma-2}dz\right)^\frac{N_\gamma-2}{N_\gamma}
  \left(\int_{{\Sigma'_\lambda}\cap B_{\varepsilon}}(|\nabla_\gamma\zeta_\varepsilon(z)|)^{N_\gamma}dz\right)^\frac{2}{N_\gamma}
  \\&\leq\left(\int_{B_{\varepsilon}}(\omega_\lambda^+(z))^\frac{2N_\gamma}{N_\gamma-2}dz\right)^\frac{N_\gamma-2}{N_\gamma}
  \left(\int_{B_{\varepsilon}}(|\nabla_\gamma\zeta_\varepsilon(z)|)^{N_\gamma}dz\right)^\frac{2}{N_\gamma},
\end{aligned}
\end{equation}
where with the constant $C>0$ in (\ref{004.16}):
\begin{equation}\label{xj16}
\begin{aligned}
  \left(\int_{B_{\varepsilon}}(|\nabla_\gamma\zeta_\varepsilon(z)|)^{N_\gamma}dz\right)^\frac{2}{N_\gamma}
  \leq\left(\Big(\frac{C}{\varepsilon}\Big)^{N_\gamma}\varepsilon^{N_\gamma}\right)^\frac{2}{N_\gamma}
  =C^2,
\end{aligned}
\end{equation}
and
\begin{equation}\label{xj17}
\begin{aligned}
  \left(\int_{B_\varepsilon}(\omega_\lambda^+(z))^\frac{2N_\gamma}{N_\gamma-2}dz\right)^\frac{N_\gamma-2}{N_\gamma}
\leq&\left(\int_{B_\varepsilon}(u(z))^\frac{2N_\gamma}{N_\gamma-2}dz\right)^\frac{N_\gamma-2}{N_\gamma}
\\\leq&D\left(\int_{B_\varepsilon}1dz\right)^\frac{N_\gamma-2}{N_\gamma}
=D\varepsilon^{N_\gamma-2}\rightarrow0,\;\;\;\;~\mbox{as}~\varepsilon\rightarrow0,
\end{aligned}
\end{equation}
since $\omega_\lambda^+(z)\leq u(z)$ and $u$ is bounded in $\R^{N+l}_+$.
Thus we obtain
\begin{equation}\label{xj18}
  \uppercase\expandafter{\romannumeral2}\rightarrow0,\;\;\;\;~\mbox{as}~\varepsilon\rightarrow0.
\end{equation}

We estimate $\uppercase\expandafter{\romannumeral1}$. In ${\Sigma'_\lambda}^+$, $\omega_\lambda(z)>0$, i.e. $u(z)>u_\lambda(z)$. So in ${\Sigma'_\lambda}^+$,
\begin{equation}\label{xj19}
\begin{aligned}
  \left(u^p(z)-u^p_\lambda(z)\right)\omega_\lambda^{+}(z)\zeta_\varepsilon^2(z)
  \leq pu^{p-1}(z)\omega_\lambda(z)\omega_\lambda^{+}(z)\zeta_\varepsilon^2(z)
  =pu^{p-1}(z)\left(\omega_\lambda^+(z)\zeta_\varepsilon(z)\right)^2.
\end{aligned}
\end{equation}
Then
\begin{equation}\label{xj20}
\begin{aligned}
  \uppercase\expandafter{\romannumeral1}
  &=\int_{\Sigma'_\lambda}\left(u^p(z)-u^p_\lambda(z)\right)\omega_\lambda^{+}(z)\zeta_\varepsilon^2(z)dz
  \\&=\int_{{\Sigma'_\lambda}^+}\left(u^p(z)-u^p_\lambda(z)\right)\omega_\lambda^{+}(z)\zeta_\varepsilon^2(z)dz
  \\&\leq\int_{{\Sigma'_\lambda}^+}pu^{p-1}(z)\left(\omega_\lambda^+(z)\zeta_\varepsilon(z)\right)^2dz
  \\&\leq p\left(\int_{{\Sigma'_\lambda}^+}(u(z))^{(p-1)\frac{N_\gamma}{2}}dz\right)^{\frac{2}{N_\gamma}}
  \left(\int_{{\Sigma'_\lambda}^+}\left(\omega_\lambda^+(z)\zeta_\varepsilon(z)
  \right)^{\frac{2N_\gamma}{N_\gamma-2}}dz\right)^{\frac{N_\gamma-2}{N_\gamma}}.
\end{aligned}
\end{equation}
 From $1+\frac{4}{N_\gamma}\leq p<\frac{N_\gamma+2}{N_\gamma-2}$ we obtain $2\leq (p-1)\frac{N_\gamma}{2}<\frac{2N_\gamma}{N_\gamma-2}$.
Since the embedding $$H^{1,2}_\gamma(\R^{N+l}_+)\hookrightarrow L^p(\R^{N+l}_+)$$ is known to be continuous for $p\in[2,2_\gamma^*]$
(see \cite{4}), we conclude that $u(z)\in L^{(p-1)\frac{N_\gamma}{2}}(\R^{N+l}_+)$. So by \eqref{004.21}, \eqref{xj15}, \eqref{xj20}, we have
\begin{equation}\label{xj21}
\begin{aligned}
   &\int_{{\Sigma'_\lambda}^+}|\nabla_\gamma\left(\omega_\lambda^{+}(z)\zeta_\varepsilon(z)\right)|^2dz
\\\leq&p\left(\int_{{\Sigma'_\lambda}^+}(u(z))^{(p-1)\frac{N_\gamma}{2}}dz\right)^{\frac{2}{N_\gamma}}
  \left(\int_{{\Sigma'_\lambda}^+}\left(\omega^+_\lambda(z)\zeta_\varepsilon(z)
  \right)^{\frac{2N_\gamma}{N_\gamma-2}}dz\right)^{\frac{N_\gamma-2}{N_\gamma}}
  \\+&\left(\int_{B_{\varepsilon}}(\omega_\lambda^+(z))^\frac{2N_\gamma}{N_\gamma-2}dz\right)^\frac{N_\gamma-2}{N_\gamma}
  \left(\int_{B_{\varepsilon}}(|\nabla_\gamma\zeta_\varepsilon(z)|)^{N_\gamma}dz\right)^\frac{2}{N_\gamma}.
\end{aligned}
\end{equation}
Letting $\varepsilon\rightarrow0$ we obtain
\begin{equation}\label{xj22}
   \int_{{\Sigma'_\lambda}^+}|\nabla_\gamma\omega_\lambda^{+}(z)|^2dz
\leq p\left(\int_{{\Sigma'_\lambda}^+}(u(z))^{(p-1)\frac{N_\gamma}{2}}dz\right)^{\frac{2}{N_\gamma}}
  \left(\int_{{\Sigma'_\lambda}^+}(\omega^+_\lambda(z)
)^{\frac{2N_\gamma}{N_\gamma-2}}dz\right)^{\frac{N_\gamma-2}{N_\gamma}}.
\end{equation}
Recall that $\left(\displaystyle\int_{\R^{N+l}}|u|^{2_\gamma^\ast}dz\right)^{\frac{2}{2_\gamma^\ast}}\leq C(N,l,\gamma)\displaystyle\int_{\R^{N+l}}|\nabla_\gamma u|^2dz$ (see \cite{3}). So, with a suitable constant $C>0$ \eqref{xj22} transforms into
\begin{equation}\label{xj23}
  \|\omega_\lambda^+(z)\|_{L^{2_\gamma^*}({{\Sigma'_\lambda}^+})}^2
  \leq C\left(\int_{{\Sigma'_\lambda}^+}(u(z))^{(p-1)\frac{N_\gamma}{2}}dz\right)^{\frac{2}{N_\gamma}}
  \|\omega_\lambda^+(z)\|_{L^{2_\gamma^*}({{\Sigma'_\lambda}^+})}^2.
\end{equation}
Since $u(z)\in L^{(p-1)\frac{N_\gamma}{2}}(\R^{N+l}_+)$, we can choose $\lambda >0$ near zero such that
\begin{equation}\label{xj24}
  C\left(\int_{{\Sigma'_\lambda}^+}(u(z))^{(p-1)\frac{N_\gamma}{2}}dz\right)^{\frac{2}{N_\gamma}}<1.
\end{equation}
From (\ref{xj23}) we conclude that $\omega_\lambda^{+}(z)=0$ in ${\Sigma'_\lambda}$ for such $\lambda$. That is for any $\lambda$ close enough to $0$, $\omega_\lambda(z)\leq0$ in ${\Sigma'_\lambda}$.

\vskip 0.2cm

\noindent\textbf{Step 2: Let} $\Lambda=\sup\{\lambda>0\;|\;\omega_\lambda(z)\leq0,~\mbox{in}~\Sigma'_\lambda \}$. \textbf{Prove that $\Lambda=+\infty$}.

\vskip 0.2cm

We derive a contradiction when assuming that $\Lambda<+\infty$. By continuity, we have $\omega_\Lambda(z)\leq0$ in $\Sigma'_\Lambda$ i.e. $u(z)\leq u_\Lambda(z)$ in $\Sigma'_\Lambda$. Recall that $\omega_\Lambda(z)$ satisfies
\begin{equation}\label{004.25}
    \begin{cases}
-\Delta_\gamma \omega_\Lambda(z)=u^p(z)-u^p_\Lambda(z)\leq0,\;\;&~ \mbox{in} ~\Sigma'_\Lambda,\\[2mm]
\omega_\Lambda(z)\leq0,  &~ \mbox{on} ~\partial\Sigma'_\Lambda.
\end{cases}
\end{equation}
However, $\omega_\Lambda(z)\equiv0$ in $\Sigma'_\Lambda$ cannot hold since $u(0,0,0)=0<u(0,0,2\Lambda)$.
By Lemma \ref{unbounded1}, we have $\omega_\Lambda(z)<0$ in $\Sigma'_\Lambda$.
\vspace{1ex}\par
If $\omega_\Lambda(z)<0$ in $\Sigma'_\Lambda$, then the plane $T_{\lambda}^{\prime}$ can be slightly moved to the right.
Recall that ${\Sigma'_\lambda}^+=\{z\in\Sigma'_\lambda\: |\;\omega_\lambda(z)>0\}$. Since ${\Sigma'_\Lambda}^+= \emptyset$, we obtain $|{\Sigma'_\lambda}^+|\rightarrow0$ as $\lambda\rightarrow\Lambda^{+}$. We test Equation \eqref{004.14} with the function $\psi(z):=\omega_\lambda^{+}(z)\zeta_\varepsilon^2(z)$ for $\lambda>\Lambda$, i.e. we consider Equation \eqref{004.18}. The same procedure as before shows:
\begin{equation}\label{xj30}
  \|\omega_\lambda^+(z)\|_{L^{2_\gamma^*}({{\Sigma'_\lambda}^+})}^2
  \leq C\left(\int_{{\Sigma'_\lambda}^+}(u(z))^{(p-1)\frac{N_\gamma}{2}}dz\right)^{\frac{2}{N_\gamma}}
  \|\omega_\lambda^+(z)\|_{L^{2_\gamma^*}({{\Sigma'_\lambda}^+})}^2.
\end{equation}
Let $\lambda >\Lambda$ be close to $\Lambda$, such that
\begin{equation}\label{xj31}
  C\left(\int_{{\Sigma'_\lambda}^+}(u(z))^{(p-1)\frac{N_\gamma}{2}}dz\right)^{\frac{2}{N_\gamma}}<1.
\end{equation}
Then we have $\omega_\lambda^{+}(z)=0$ in ${\Sigma'_\lambda}$, i.e. $\omega_\lambda(z)\leq0$ in ${\Sigma'_\lambda}$. This contradicts the definition of $\Lambda$ and $\Lambda=+\infty$ follows.
\vskip 0.2cm

\noindent\textbf{Step 3: Reduce to the statement of Lemma \ref{th1.2} and end of the proof}.

\vskip 0.2cm

According to Step 1 and Step 2, we know that $u(x,\widetilde{y}',y_l)$ is increasing with respect to the variable $y_l$. On the other hand, $u$ is bounded in $\R^{N+l}_+$ by assumption. Therefore, there exists $v(x,\widetilde{y}')\geq0$ in $\R^{N+l-1}_+$ such that
\begin{equation}\label{004.29}
  \lim_{y_l\rightarrow+\infty}u(x,\widetilde{y}',y_l)=v(x,\widetilde{y}')\geq0.
\end{equation}
Clearly, $v(x,\widetilde{y}')$ satisfies
\begin{equation}\label{004.30}
  -\Delta_\gamma v(x,\widetilde{y}')=v^p(x,\widetilde{y}'),\;\; v(x,\widetilde{y}')\geq0,~ \mbox{in} ~\R^{N+l-1}.
\end{equation}
Since $p<\frac{N_\gamma+2}{N_\gamma-2}<\frac{N_\gamma+1-\gamma}{N_\gamma-3-\gamma}$, by Lemma \ref{th1.2} in the case of $\R^{N+l-1}$, we find that $v(x,\widetilde{y}')=0$. Combining this observation with the nonnegativity of $u$  and its property of being increasing in $y_l$ it follows that $u(z)=0$. This completes the proof
of Theorem \ref{th1.3}.
\end{proof}

\section{A Priori estimates}

In this section, we prove a prior estimates for positive solutions of the Equation \eqref{002.1} by employing blow up analysis in Theorem \ref{th1.4}. Before that, we give the Schauder estimates for degenerate elliptic equations in \cite{book} which we use in our proof.

\begin{Prop}\label{Schauder}\textbf{(Theorem 11.1 (Regularity), \cite{book})}\\
Let $L=\displaystyle\sum_{i=1}^q X_i^2$ be a H\"{o}rmander operator without drift in $\Omega\subset\R^N$. Suppose that $u\in\mathcal{D}'(\Omega)$ satisfies $Lu=f$ with $f\in C^\theta(\overline{\Omega})$. Then for any $\Omega'\subset\subset\Omega$, we have
\begin{equation}\label{wsgs11}
  \|u\|_{2+\theta,\Omega'}\leq C\big{(}\|u\|_{L^\infty(\Omega)}+\|f\|_{\theta,\overline{\Omega}}\big{)}.
\end{equation}
The positive constant $C$ depends on $L$, $\Omega$, $N$, $\theta$. Here for any $0<\theta<1$, we have used the following notations:
\begin{equation}\label{wsgs12}
  [u]_{\theta,K}=\sup_{x,y\in K, x\neq y}\frac{|u(x)-u(y)|}{(d(x,y))^\theta},
\end{equation}
\begin{equation}\label{wsgs13}
  \|u\|_{2+\theta,K}=\sum_{|J|\leq2}\left(\sup_K|X^J u|+[X^J u]_{\theta,K}\right).
\end{equation}
Let $1\leq j_i\leq q$ and $J=(j_1,\cdots,j_k)$ denotes a multi-index with length $|J|=k$, $0\leq k\leq q$. We adopt the notation $X^J=X_{j_1}X_{j_2}\cdots X_{j_k}$ for $|J|=k$, and $X^J=id$ for $|J|=0$.
\end{Prop}

\begin{Rem}\label{Sch}
The Grushin operator $\mathcal{L}=\Delta_{\gamma}=\Delta_x+|x|^{2\gamma}\Delta_y$ with $(x,y)\in\R^{N+l}$ is a H\"{o}rmander operator when $\gamma=m$ and $m\in\N$. Thus we can apply Proposition \ref{Schauder} to
$u\in C^\theta(\overline{\Omega})$, where $\Omega\subset\R^{N+l}$ is an open domain.
\end{Rem}

Next, we prove Theorem \ref{th1.4} and we apply the Schauder estimate \eqref{wsgs11} in order to obtain  a priori estimates.

\begin{proof}[\textbf{Proof of Theorem \ref{th1.4}:}]

Let $A$ be the set {of solutions to Equation \eqref{002.1} in $H_{\gamma}^{1,2}(\Omega) \cap C^{\theta}(\overline{\Omega})$}.
We show that there exists $C>0$ such that for any $u\in A$
\begin{equation}\label{wsgs20}
  \displaystyle\sup_{z\in\overline{\Omega}}u(z)\leq C.
\end{equation}

We prove this statement by contradiction. Assume that for any $k>0$, there exists $u_k\in A$ and maximum value points $q_k=(t_k,r_k)\in\overline{\Omega}$, such that
\begin{equation}\label{cwf2}
  M_k=\sup_{z\in\overline{\Omega}}u_k(z)=u_k(q_k)\rightarrow+\infty, ~\mbox{as}~k\rightarrow\infty.
\end{equation}
Since $\overline{\Omega}$ is compact, we may assume that there exists $q=(t,r)\in\overline{\Omega}$, such that $q_k\rightarrow q$ in $\overline{\Omega}$, as $k\rightarrow\infty$. In the following proof we distinguish the cases $q\in\Omega$ and $q\in\partial\Omega$.
\vspace{1ex}\\
\noindent \textbf{Case 1:} We will show that the case of $q\in\Omega$ can be reduced to the statement of Lemma \ref{th1.2}. Let $2d=d(q,\partial\Omega)$ be the $d$-distance of $q$ to the boundary of $\Omega$ and by $\widetilde{B}_R(0)$ denote the ball of radius $R$ with respect to the metric \eqref{006.6} centered at $0$.

Let $\lambda_k$ be a sequence of positive numbers such that $\lambda_k^{\frac{2}{p-1}}M_k=1$. Since $\frac{2}{p-1}>0$ and $M_k\rightarrow+\infty$ as $k\rightarrow\infty$,
we conclude that $\lambda_k\rightarrow0$ as $k\rightarrow\infty$. We transform the variables $z=(x,y)$ to a new pair of variables
$\widetilde{z}=(\widetilde{x},\widetilde{y})$ by defining:
\begin{equation}\label{002.3}
  \widetilde{x}=\frac{x-t_k}{\lambda_k},\;\;\widetilde{y}=\frac{y-r_k}{\lambda_k^{1+\gamma}}.
\end{equation}
Then we define the blow up sequence
\begin{equation}\label{002.4}
  v_k(\widetilde{z})=\lambda_k^{\frac{2}{p-1}}u_k(\lambda_k\widetilde{x}+t_k,\lambda_k^{1+\gamma}\widetilde{y}+r_k).
\end{equation}
Note that $v_k(\widetilde{z})$ is well defined in $\widetilde{B}_{\frac{d}{\lambda_k}}(0)$ for $k$ large enough and by $\lambda_k^{\frac{2}{p-1}}M_k=1$ we have
\begin{equation}\label{002.5}
  \sup_{\widetilde{B}_{\frac{d}{\lambda_k}}(0)}v_k(\widetilde{z})=1=v_k(0).
\end{equation}
Moreover, $v_k(\widetilde{z})$ satisfies
\begin{equation}\label{002.6}
\begin{aligned}
  &-\left(\lambda_k^{2\gamma}\Delta_{\widetilde{x}}v_k(\widetilde{z})
  +|\lambda_k\widetilde{x}+t_k|^{2\gamma}
  \Delta_{\widetilde{y}}v_k(\widetilde{z})\right)
  \\=&\lambda_k^{\frac{2p}{p-1}+2\gamma}f\Big{(}(\lambda_k\widetilde{x}+t_k,\lambda_k^{1+\gamma}\widetilde{y}+r_k),
  \lambda_k^{-\frac{2}{p-1}}
  v_k(\widetilde{z})\Big{)},\;~\mbox{in}~\widetilde{B}_{\frac{d}{\lambda_k}}(0).
\end{aligned}
\end{equation}
By assumption \eqref{002.2}, i.e.
\begin{equation}\label{002.7}
\lim_{t\rightarrow+\infty}t^{-p}(f(z,t)-h(z)t^p)=0,
\end{equation}
we have
\begin{equation}\label{002.8}
  \lim_{k\rightarrow\infty}\lambda_k^{\frac{2p}{p-1}}f\Big{(}(\lambda_k\widetilde{x}+t_k,\lambda_k^{1+\gamma}\widetilde{y}+r_k),\lambda_k^{-\frac{2}{p-1}}v_k(\widetilde{z})\Big{)}
  -h(\lambda_k\widetilde{x}+t_k,\lambda_k^{1+\gamma}\widetilde{y}+r_k)v_k^p(\widetilde{z})=0.
\end{equation}

So, as $k \rightarrow \infty$, the sequence $v_k(\widetilde{z}), k \in \mathbb{N}$ satisfies in $\widetilde{B}_{\frac{d}{\lambda_k}}(0)$:
\begin{equation}\label{cwf3}
  -\left(\lambda_k^{2\gamma}\Delta_{\widetilde{x}}v_k(\widetilde{z})
  +|\lambda_k\widetilde{x}+t_k|^{2\gamma}
  \Delta_{\widetilde{y}}v_k(\widetilde{z})\right)
  =\lambda_k^{2\gamma}h(\lambda_k\widetilde{x}+t_k,\lambda_k^{1+\gamma}\widetilde{y}+r_k)v_k^p(\widetilde{z})+ \lambda_k^{2\gamma}o(1).
\end{equation}
Certainly, $v_k(\widetilde{z})$ satisfies Equation (\ref{cwf3}) on any given ball $\widetilde{B}_R(0)$, with $
\Omega\subset\widetilde{B}_R(0)\subset\widetilde{B}_{\frac{d}{\lambda_k}}(0)$.
Since $u_k(z)\in H^{1,2}_\gamma(\Omega)\cap C^\theta(\overline{\Omega})$, we conclude that $v_k(\widetilde{z})\in H^{1,2}_\gamma(\widetilde{B}_R(0))\cap C^\theta(\overline{\widetilde{B}_R(0)})$. By \eqref{002.5}, we have
\begin{equation}\label{cwff}
  \sup_{\widetilde{B}_R(0)}v_k(\widetilde{z})=1.
\end{equation}
Recall that $H^{1,2}_\gamma(\widetilde{B}_R(0))$ is a Hilbert space with compact embedding $H^{1,2}_\gamma(\widetilde{B}_R(0))\hookrightarrow L^2(\widetilde{B}_R(0))$ (see \cite{4}). {Note that \eqref{002.6}, properties of $f$ as well as \eqref{cwff} imply that $(v_k(\tilde{z}))_k$ is bounded in $H^{1,2}_\gamma(\widetilde{B}_R(0))$.}
By passing to a subsequence, if necessary, we may assume that there exists $v(\widetilde{z})\in H^{1,2}_\gamma(\widetilde{B}_R(0))$, such that
\begin{equation}\label{cwf4}
  v_k(\widetilde{z})\rightharpoonup v(\widetilde{z}),\;\;~\mbox{weakly in}~H^{1,2}_\gamma(\widetilde{B}_R(0)),
\end{equation}
\begin{equation}\label{cwf5}
  v_k(\widetilde{z})\rightarrow v(\widetilde{z}),\;\;~\mbox{strongly in}~L^2(\widetilde{B}_R(0)),
\end{equation}
\begin{equation}\label{cwf6}
  v_k(\widetilde{z})\rightarrow v(\widetilde{z}),\;\;~\mbox{a.e. on}~\widetilde{B}_R(0).
\end{equation}
For any $\varphi(\widetilde{z})\in C_0^\infty(\widetilde{B}_R(0))$, from \eqref{cwf3} we have as $k \rightarrow \infty$:
\begin{equation}\label{cwf7}
\begin{aligned}
  &\int_{\widetilde{B}_R(0)}\nabla_{\widetilde{x}}v_k(\widetilde{z})\cdot\nabla_{\widetilde{x}}\varphi(\widetilde{z})
  +|\widetilde{x}+\lambda_k^{-1} t_k|^{2\gamma}
  \nabla_{\widetilde{y}}v_k(\widetilde{z})\cdot\nabla_{\widetilde{y}}\varphi(\widetilde{z})d\widetilde{z}
  \\=&\int_{\widetilde{B}_R(0)}h(\lambda_k\widetilde{x}+t_k,\lambda_k^{1+\gamma}\widetilde{y}+r_k)
  v_k^p(\widetilde{z})
  \varphi(\widetilde{z})d\widetilde{z}+o(1).
\end{aligned}
\end{equation}
\vspace{1mm}\par
First assume that $\lim_{k \rightarrow \infty} \lambda_k^{-1} |t_k| =0$. Using
\eqref{cwf4}, \eqref{cwf6} together with Lebesgues Dominated Convergence Theorem implies
\begin{equation}\label{cwf8}
  \int_{\widetilde{B}_R(0)}\nabla_\gamma v(\widetilde{z})\cdot\nabla_\gamma\varphi(\widetilde{z})d\widetilde{z}
  =\int_{\widetilde{B}_R(0)}h(q)v^p(\widetilde{z})
  \varphi(\widetilde{z})d\widetilde{z}, 
\end{equation}
i.e.
\begin{equation}\label{002.10}
  -\Delta_\gamma v(\widetilde{z})=h(q)v^p(\widetilde{z}),\;~\mbox{in}~H^{1,2}_\gamma(\widetilde{B}_R(0)).
\end{equation}
Assume that $\lim \sup_k \lambda_k^{-1} |t_k|  \in (0, \infty)$.  By passing to a subsequence if necessary we can assume that
$\lim_k \lambda_k^{-1} t_k = \eta \in \mathbb{R}^N$ exists. Then
it follows from \eqref{cwf7} that
\begin{equation}\label{cwf8a-new}
\int_{\widetilde{B}_R(0)} \nabla_{\widetilde{x}} v(\widetilde{z}) \cdot \nabla_{\widetilde{x}} \varphi(\widetilde{z}) + |\widetilde{x}+ \eta|^{2 \gamma} \nabla_{\widetilde{y}} v(\widetilde{z}) \cdot \nabla_{\widetilde{y}}\varphi(\widetilde{z}) \:  d\widetilde{z}= \int_{\widetilde{B}_R(0)}h(q)v^p(\widetilde{z})
  \varphi(\widetilde{z})d\widetilde{z}.
\end{equation}
Finally, assume that $\lim \sup_k \lambda_k^{-1} |t_k| =\infty$. In this case  \eqref{cwf7} implies that
\begin{equation}\label{cwf8a}
  \int_{\widetilde{B}_R(0)}\nabla_{\widetilde{y}} v(\widetilde{z})\cdot\nabla_{\widetilde{y}}\varphi(\widetilde{z})d\widetilde{z}
  =0. 
\end{equation}

On the other hand, by the H\"older continuity of the function $h(z)$, \eqref{002.5}, \eqref{cwf3} and the Schauder estimate \eqref{wsgs11}, for any given $R$ such that $\widetilde{B}_R(0)\subset\widetilde{B}_{\frac{d}{\lambda_k}}(0)$, we find uniform bounds for $\|v_k\|_{2+\theta,\overline{\widetilde{B}_R(0)}}$. After passing to a subsequence if necessary we have uniformly on compact sets:
\begin{equation}\label{cwf9}
  v_k(\widetilde{z})\rightarrow v(\widetilde{z}),
\end{equation}
and
\begin{equation}\label{wys1}
  -\Delta_\gamma v_k(\widetilde{z})\rightarrow-\Delta_\gamma v(\widetilde{z}).
\end{equation}
Let $k\rightarrow\infty$ in \eqref{cwf3} and consider the equation in $\widetilde{B}_R(0)$, by continuity, if $\lim_{k \rightarrow \infty} \lambda_k^{-1}t_k=0$, we have
\begin{equation}\label{cwf10}
  -\Delta_\gamma v(\widetilde{z})=h(q)v^p(\widetilde{z}),\;~\mbox{in}~C^\theta(\overline{\widetilde{B}_R(0)}).
\end{equation}
Assume that $\lim_k \lambda_k^{-1} t_k = \eta \in \mathbb{R}^N$ exists, then (\ref{cwf8a-new}) holds and therefore
\begin{equation}\label{cwf10-new}
-\left( \Delta_{\widetilde{x}} v(\widetilde{z}) + |\widetilde{x}+ \eta|^2 \Delta_{\widetilde{y}} v(\widetilde{z}) \right)= h(q) v^p (\widetilde{z})
\end{equation}
in $C^\theta(\overline{\widetilde{B}_R(0)})$. Finally, if we assume that $\lim \sup_k \lambda_k^{-1} |t_k| =\infty$, then Equation \eqref{cwf3} implies that
\begin{equation}\label{cwf10aa-New}
\Delta_{\widetilde{y}}v(\widetilde{z})=0
\end{equation}
 in $C^\theta(\overline{\widetilde{B}_R(0)})$.

Combining \eqref{002.10} with \eqref{cwf10}, \eqref{cwf8a-new} with \eqref{cwf10-new} and \eqref{cwf8a} with \eqref{cwf10aa-New}
shows that $v(\widetilde{z})$ satisfies one of the following
three equations in $H^{1,2}_\gamma(\widetilde{B}_R(0))\cap C^\theta(\overline{\widetilde{B}_R(0)})$:
\begin{align}
 -\Delta_\gamma v(\widetilde{z})&=h(q)v^p(\widetilde{z}), \hspace{2ex} \mbox{or}\label{cwf11} \\
-\left(\Delta_{\widetilde{x}} v(\widetilde{z}) + |\widetilde{x}+ \eta|^2 \Delta_{\widetilde{y}} v(\widetilde{z})\right) &= h(q) v^p (\widetilde{z}) , \hspace{2ex} \mbox{or}
 \label{cwf11-new}\\
 \Delta_{\widetilde{y}}v(\widetilde{z})&=0. \label{cwf11a}
 \end{align}
From \eqref{002.4}, we know that
\begin{equation}\label{cwf12}
  v_k(0)=\lambda_k^{\frac{2}{p-1}}u_k(q_k)=1.
\end{equation}
As $k\rightarrow\infty$ and by H\"older continuity we conclude that
\begin{equation}\label{cwf13}
  v(0)=1.
\end{equation}

We claim that $v(\widetilde{z})$ is well defined in all of $\R^{N+l}$. To show this, we consider a ball $\widetilde{B}_{R'}(0)$ which satisfies $\widetilde{B}_R(0)\subset\widetilde{B}_{R'}(0)\subset\widetilde{B}_{\frac{d}{\lambda_k}}(0)$. Repeating the above argument with respect to $\widetilde{B}_{R'}(0)$, the sequence $v_k$ has a convergent subsequence $v_{k_j}$, such that
\begin{equation}\label{cwf14}
  v_{k_j}\rightarrow v'\;\;~\mbox{on}~\widetilde{B}_{R'}(0),
\end{equation}
and $v'$ satisfies $v'|_{\widetilde{B}_R(0)}=v$. By unique continuation $v$ can be extended as a solution to the entire space $\R^{N+l}$ which we will still
denote by $v$.
In fact since $2d=d(q,\partial\Omega)$ is the $d$-distance of $q$ to $\partial\Omega$, for any $z=(x, y) \in \partial\Omega$ one has $d(q,z)=d((t,r),(x,y))\geq2d$.
If $k$ is sufficiently large such that $d(q,q_k)=d((t-t_k, r-r_k),0) \leq d$, then by \eqref{002.3} we have
\begin{equation*}
\begin{aligned}
  2d\leq d\Big{((t,r),(\lambda_k\widetilde{x}+t_k,\lambda_k^{1+\gamma}\widetilde{y}+r_k})\Big{)}
  \leq&d((t-t_k,r-r_k), 0)+ d\big{(}(\lambda_k \tilde{x}, \lambda_k^{1+ \gamma} \tilde{y}),0\big{)}
  \\=&d(q,q_k)+\lambda_kd((\tilde{x}, \tilde{y}),0)
  \\\leq&d+\lambda_kd((\tilde{x}, \tilde{y}),0).
\end{aligned}
\end{equation*}
Set $\widetilde{z}=(\widetilde{x},\widetilde{y})$, from above inequality we obtain $d(\widetilde{z},0)\geq\frac{d}{\lambda_k}$. By \eqref{002.4}, i.e. $ v_k(\widetilde{z})=\lambda_k^{\frac{2}{p-1}}u_k(z)$.
Since $u_k(z)=0$ when $z\in\partial\Omega$, then it holds $v_k(\widetilde{z})=0$ when $d(\widetilde{z},0)\geq\frac{d}{\lambda_k}$.
Let $\lambda_k\rightarrow0$, we have $v_k(\widetilde{z})=0$ as $d(\widetilde{z},0)\rightarrow\infty$.
Therefore, the limit $v$ of $v_k$ vanishes at infinity.

In particular, note that in the cases \eqref{cwf11} and \eqref{cwf11-new} the function $v_{\eta}(\widetilde{z}):=v(\widetilde{x}- \eta, \widetilde{y})$ for some $\eta \in \mathbb{R}^N$ satisfies
\begin{equation}\label{002.13}
\begin{cases}
-\Delta_\gamma v_{\eta}(\widetilde{z})=h(q)v_{\eta}^p(\widetilde{z}),\;\;v_{\eta}(\widetilde{z})>0,&~ \mbox{in} ~H^{1,2}_\gamma(\R^{N+l})\cap C^\theta(\R^{N+l}),\\[2mm]
v_{\eta}(\eta,0)=v(0)=1.
\end{cases}
\end{equation}
In the case \eqref{cwf11a} the function $v(\widetilde{z})$ satisfies:
\begin{equation}\label{002.13a}
\begin{cases}
\Delta_{\widetilde{y}}v(\widetilde{z})=0,\;\;v(\widetilde{z})>0,&~ \mbox{in} ~H^1(\R^{N+l})\cap C^\theta(\R^{N+l}),\\[2mm]
v(0)=1.
\end{cases}
\end{equation}
Note that \eqref{002.13a} states that $v(\widetilde{x}, \cdot )$ is a positive harmonic function on $\mathbb{R}^l$ for each $\widetilde{x}\in \mathbb{R}^N$ and by well-known results it must be constant. Since $v$ vanishes at infinity, it follows that $v(\widetilde{x}, \cdot ) \equiv 0$ for a.e. $\widetilde{x} \in \mathbb{R}^l$. Hence $v \equiv 0 $ by continuity which contradicts the assumption $v(0)=1$.

\vspace{1ex}\par
In order to derive a contradiction from \eqref{002.13} we perform a stretching transformation such that the constant term $h(q)$ in \eqref{002.13} vanishes. Let
\begin{equation}\label{002.14}
  \widetilde{x}:=\frac{x'}{\sqrt{h(q)}},\;\;\widetilde{y}:=\frac{y'}{\left(\sqrt{h(q)}\right)^{1+\gamma}}.
\end{equation}
We write
\begin{equation}\label{002.15}
  v_{\eta}(\widetilde{z})=v_{\eta}(\widetilde{x},\widetilde{y})=v_{\eta}\left(\frac{x'}{\sqrt{h(q)}},\frac{y'}
  {\left(\sqrt{h(q)}\right)^{1+\gamma}}\right)=:v'_{\eta}(x',y')=v'_{\eta}(z').
\end{equation}
Then $v_{\eta}'(z')$, assuming one of the cases \eqref{cwf11} or  \eqref{cwf11-new},  satisfies:
\begin{equation}\label{002.16}
\begin{cases}
-\Delta_\gamma v_{\eta}'(z')=v_{\eta}'^p(z'),\;\;v_{\eta}'(z')>0,&~ \mbox{in} ~H^{1,2}_\gamma(\R^{N+l})\cap C^\theta(\R^{N+l}),\\[2mm]
v_{\eta}'(\eta \sqrt{h(q)},0)=1.
\end{cases}
\end{equation}
But according to Lemma \ref{th1.2} and \cite[Theorem 1.2]{21}, equation $-\Delta_\gamma v'=v'^p$
only has the trivial solution $v'\equiv0$ contradicting $v_{\eta}'(\eta \sqrt{h(q)},0)=1$ in the first case.
Hence we obtain that $q\in\Omega$ is not valid. This completes the proof of Case 1.

\vskip 0.1cm
\noindent \textbf{Case 2:} We will show that the case of $q\in\partial\Omega$ can be reduced to the statement of Theorem \ref{th1.3}. Since $\partial\Omega$ is $C^1$, we can straighten $\partial\Omega$ in a neighborhood of $q$ by a non-singular $C^1$ change of coordinates.

Let $y_l=\psi(z)$, where $\psi$ is a $C^1$ defining function for $\partial\Omega$ and $z=(x,y)$. Define a new coordinates system as follows:
\begin{eqnarray*}
  \widehat{x_i} &=& x_i,\;\;i=1,\cdots,N, \\
  \widehat{y_j} &=& y_j,\;\;j=1,\cdots,l-1, \\
  \widehat{y_l}&=& y_l-\psi(z).
\end{eqnarray*}
{The solution to \eqref{002.1} in the new coordinate system is still denoted as $u$.} For convenience, we use the notation $z=(x,y)$ instead of $\widehat{z}=(\widehat{x},\widehat{y})$.
Then $q\in\partial\Omega$ is located in the hyperplane $y_l=0$. We define $\lambda_k$, $v_k(\widetilde{z})$ as in Case 1.
Let $d_k:=d(q_k,\partial\Omega)$ be the distance of $q_k$ to the boundary measured with respect to the metric $d$ in \eqref{006.6}. Then $d_k=q_k\cdot e_{y_l}$, where $e_{y_l}$ is the unit vector in $y_l$ direction. For any $\delta >0$ and $k$ large enough, $v_k(\widetilde{z})$ is well defined in
$\widetilde{B}_{\frac{\delta}{\lambda_k}}(0)\cap\{\widetilde{z}\;|\;\widetilde{y_l}>-\frac{d_k}{\lambda_k^{1+\gamma}}\}$. Moreover, by assumption \eqref{002.2}, the function $v_k(\widetilde{z})$ satisfies
\begin{equation}\label{002.17}
\begin{split}
    -\left(\lambda_k^{2\gamma}\Delta_{\widetilde{x}}v_k(\widetilde{z})
  +|\lambda_k\widetilde{x}+t_k|^{2\gamma}
  \Delta_{\widetilde{y}}v_k(\widetilde{z})\right)
  =\lambda_k^{2\gamma}h(\lambda_k\widetilde{x}+t_k,\lambda_k^{1+\gamma}\widetilde{y}+r_k)v_k^p(\widetilde{z})+ \lambda_k^{2\gamma}o(1),
\end{split}
\end{equation}
in
\begin{equation*}
\widetilde{B}_{\frac{\delta}{\lambda_k}}(0)\cap\Big{\{}\widetilde{z} |\;\widetilde{y_l}>-\frac{d_k}{\lambda_k^{1+\gamma}}\Big{\}},
\end{equation*}
and it holds
\begin{equation}\label{002.18}
  \sup_{\widetilde{B}_{\frac{\delta}{\lambda_k}}(0)\cap\big{\{}\widetilde{z}\:|\;\widetilde{y_l}>-\frac{d_k}{\lambda_k^{1+\gamma}}\big{\}}}v_k(\widetilde{z})=1.
\end{equation}

We claim that $\frac{d_k}{\lambda_k^{1+\gamma}}$ is bounded from below. In fact, by the $C^\theta$-regularity we have
\begin{equation}\label{002.19}
   \big{|}v_k(0,0,0)-v_k\big{(}0,0,-\frac{d_k}{\lambda_k^{1+\gamma}}\big{)}\big{|}
\leq L\Big{(}\frac{d_k}{\lambda_k^{1+\gamma}}\Big{)}^\theta,
\end{equation}
where $L>0$ is a constant.
On the other hand by \eqref{002.4}, we have
\begin{equation}\label{bbs1}
  v_k(0,0,0)=\lambda_k^{\frac{2}{p-1}}u_k(t_k,r_k)=\lambda_k^{\frac{2}{p-1}}u_k(q_k)=\lambda_k^{\frac{2}{p-1}}M_k=1,
\end{equation}
\begin{equation}\label{bbs2}
  v_k\big{(}0,0,-\frac{d_k}{\lambda_k^{1+\gamma}}\big{)}=\lambda_k^{\frac{2}{p-1}}u_k(t_k,r_k-d_k e_l)
  =\lambda_k^{\frac{2}{p-1}}u_k\big|_{\partial\Omega}=0.
\end{equation}
Thus, we obtain
\begin{equation}\label{002.20}
  \big{|}v_k(0,0,0)-v_k\big{(}0,0,-\frac{d_k}{\lambda_k^{1+\gamma}}\big{)}\big{|}=1.
\end{equation}
Combing \eqref{002.19} and \eqref{002.20}, and letting $k\rightarrow\infty$, we obtain
\begin{equation}\label{002.21}
  \frac{d_k}{\lambda_k^{1+\gamma}}\geq \Big{(}\frac{1}{L}\Big{)}^{\frac{1}{\theta}},
\end{equation}
that is $\frac{d_k}{\lambda_k^{1+\gamma}}$ is bounded from below.

Assume that $\frac{d_k}{\lambda_k^{1+\gamma}}$ is not bounded from above as $k \rightarrow \infty$. After passing to a subsequence, if necessary, we may assume that $\frac{d_k}{\lambda_k^{1+\gamma}}\rightarrow+\infty$ as $k\rightarrow\infty$. Then $v_k(\widetilde{z})$ is well defined in a ball $\widetilde{B}_R(0)$ of any fixed radius $R>0$ and satisfies
\begin{equation}\label{002.22}
\begin{cases}
-\left(\Delta_{\widetilde{x}}v_k(\widetilde{z})
  +|\widetilde{x}+\lambda_k^{-1}t_k|^{2\gamma}
  \Delta_{\widetilde{y}}v_k(\widetilde{z})\right)
  =h(\lambda_k\widetilde{x}+t_k,\lambda_k^{1+\gamma}\widetilde{y}+r_k)v_k^p(\widetilde{z})
  +o(1),\\[2mm]
v_k(0)=1, \hspace{1ex} \mbox{\it and} \hspace{1ex} v_k(\widetilde{z})>0. &
\end{cases}
\end{equation}
Under these conditions we can repeat the discussion in Case 1, and we obtain that $q\in\partial\Omega$ with $\frac{d_k}{\lambda_k^{1+\gamma}}$ being unbounded is not possible.

Assume now that  $\frac{d_k}{\lambda_k^{1+\gamma}}$ is bounded from above as $k \rightarrow \infty$. After passing to a subsequence, if necessary, we may assume that $\frac{d_k}{\lambda_k^{1+\gamma}}\rightarrow s$ as $k\rightarrow\infty$ for some $s>0$. By arguments as in Case 1, and notation there we obtain
\begin{equation}\label{cwf31}
\begin{cases}
-\Delta_\gamma v_{\eta}(\widetilde{z})=h(q)v_{\eta}^p(\widetilde{z}),\;\;v_{\eta}(\widetilde{z})>0,&~ \mbox{in} ~\{y_l>-s\},\\[2mm]
v_{\eta}(\widetilde{z})=0,&~ \mbox{on} ~\{y_l=-s\},
\end{cases}
\end{equation}
or
\begin{equation}\label{cwf31a}
\begin{cases}
\Delta_{\widetilde{y}}v(\widetilde{z})=0,\;\;v(\widetilde{z})>0,&~ \mbox{in} ~\{y_l>-s\},\\[2mm]
v(\widetilde{z})=0,&~ \mbox{on} ~\{y_l=-s\}.
\end{cases}
\end{equation}
According to \cite[Theorem 7.22]{ABR} the problem \eqref{cwf31a} only has the solutions $v(\tilde{z}) = c(y_l+s)$ where $c>0$. Since $v$ is bounded by
construction we obtain  a contradiction.
Hence \eqref{cwf31} must hold.
By applying a stretching transformation as in Case 1, we end up with the set up of Theorem \ref{th1.3}. However, according to Theorem \ref{th1.3}
\eqref{cwf31}
only has the solution $v(\widetilde{z})\equiv0$, which contradicts with $v(0)=1$. This implies that $q\in\partial\Omega$ with $\frac{d_k}{\lambda_k^{1+\gamma}}$ being bounded is not possible either. This completes the proof of Case 2.

Therefore, we conclude that $q\in\partial\Omega$ is not valid. Hence we have reached a contradiction to the assumption that (\ref{wsgs20}) was not true which completes the proof.
\end{proof}

\section{Existence of solutions}

In this section, we prove the existence of solutions for the Equation \eqref{002.1} by applying a topological degree theorem in Theorem \ref{th1.5}. We start by presenting the relevant definitions in Definition \ref{zhui} and the Leray Schauder degree theorem in Proposition \ref{lemma2}.
\vspace{2mm}

\begin{Def}\label{zhui}\textbf{(Cone and OBS, \cite{8})}\\
A cone $P$ in a Banach space $X$ is a closed subset of $X$ such that:
\begin{itemize}
\item[(i)]  if $x,y\in P$ and $\alpha,\beta>0$, then $\alpha x+\beta y\in P$;
\item[(ii)] if $x\in P$ and $x\neq0$, then $-x\notin P$.
\end{itemize}
A cone $P$ induces a partial ordering on $X$ by $x\leq y$ iff $y-x\in P$. We call $(X,P)$ an ordered Banach space $(OBS)$.
\end{Def}

\begin{Prop}\label{lemma2}\textbf{(Leray Schauder degree theorem, \cite{19})}\\
Suppose that $(X,P)$ is an $OBS$, $\varphi: P\rightarrow P$ is a compact map such that $\varphi(0)=0$. Assume that there exist numbers $r$, $R$ such that $0<r<R$ and a vector $v\in P\setminus\{0\}$ such that:
\begin{itemize}
\item[(i)] $x\neq t\varphi(x)$ for $0\leq t\leq1$ and $\|x\|=r$;
\item[(ii)] $x\neq\varphi(x)+tv$ for $t\geq0$ and $\|x\|=R$.
\end{itemize}
\noindent
If $U=\{x\in P|\;r<\|x\|<R\}$, then one has $deg(\varphi,U,0)=-1$. In particular, $\varphi$ has a fixed point in $U$.
\end{Prop}

\begin{Rem}\label{rem3}\textbf{(\cite{19})}\\
Condition \textup{(i)} is replaced by \textup{(i')}: if there exists a bounded linear operator $A: X\rightarrow X$, such that $A(P)\subset P$ and $A$ has spectral radius strictly less than 1 and
$$\varphi(x)\leq A(x) \hspace{2ex} \mbox{\it  for}\hspace{2ex}  x\in P, \;\; \|x\|=r.$$
Condition \textup{(ii)} is replaced by \textup{(ii')}: there is a compact map $F:\overline{B}_R(0)\times[0,\infty)\rightarrow P$, where $\overline{B}_{R}(0)$ is the closed ball in $X$ of radius $R>0$ and centered in zero, such that

\;\; \textup{(1)} $F(x,0)=\varphi(x)$ for $\|x\|=R$,

\;\; \textup{(2)} $F(x,t)\neq x$ for $\|x\|=R$ and $t\in[0,+\infty)$,

\;\; \textup{(3)} $F(x,t)=x$ has no solution for $x\in\overline{B}_R(0)$ and $t\geq t_0$ for some $t_0$ large enough.
\end{Rem}

\begin{proof}[\textbf{Proof of Theorem \ref{th1.5}:}]

According to Propsition \ref{lemma2}, let $X=L^2(\Omega)$ and note that
\begin{equation}\label{cone}
  {P_{\theta}=  \big{\{}u\in H^{1,2}_\gamma(\Omega)|\;u(z)\geq0,\;in \;\Omega\big{\}}\cap C^{\theta}(\overline{\Omega})}
\end{equation}
defines a cone in $L^2(\Omega)$. {According to \cite{hongge},} by using the weighted Poincar\'{e} inequality and the Lax-Milgram theorem, it is easy to deduce that the inverse of the degenerate elliptic operator $-\Delta_\gamma$
defined on $H_{\gamma}^{1,2}(\Omega)$ with Dirichlet boundary conditions exists. It defines a continuous linear operator
\begin{equation*}
(-\Delta_\gamma)^{-1}: L^2(\Omega)\rightarrow H^{1,2}_\gamma(\Omega).
\end{equation*}
Furthermore, since $H^{1,2}_\gamma(\Omega)$ is compactly embedded into $L^2(\Omega)$ we conclude that $(-\Delta_\gamma)^{-1}$ is a compact operator on $L^2(\Omega)$.  According to the spectral theory of compact operators, the degenerate elliptic operator $-\Delta_\gamma$ has positive discrete eigenvalues, denoted as $\{\lambda_k\}_{k=1}^\infty$, satisfying $0<\lambda_1\leq \lambda_2\leq\lambda_3\leq\cdots\leq\lambda_k\leq\cdots$ and $\lambda_k\rightarrow+\infty$ as $k\rightarrow+\infty$. If $\{\phi_k\}_{k=1}^\infty$
is a sequence of eigenfunctions corresponding to the eigenvalues $\{\lambda_k\}_{k=1}^\infty$ with Dirichlet zero boundary condition, i.e.
\begin{equation}\label{eigenvalue}
\begin{cases}
-\Delta_\gamma \phi_k=\lambda_k\phi_k,  &~ \mbox{in} ~\Omega,\\[2mm]
\phi_k=0,  &~ \mbox{on} ~\partial\Omega,
\end{cases}
\end{equation}
then $\{\phi_k\}_{k=1}^\infty$ forms an orthogonal basis for the space $H^{1,2}_\gamma(\Omega)$ as well as for $L^2(\Omega)$.
Suppose that $u\in H^{1,2}_{\gamma}(\Omega)$ satisfies \eqref{002.1} and define
\begin{equation}\label{001.2}
  u=(-\Delta_\gamma)^{-1}f(z,u)=:Lf(z,u).
\end{equation}
Then $L= (-\Delta_\gamma)^{-1}: L^2(\Omega)\rightarrow L^2(\Omega)$ is a compact operator and
\begin{equation}\label{001.6}
  \varphi(u):=Lf(z,u),
\end{equation}
defines a compact map from $P_\theta\rightarrow P_\theta$. In fact, $\varphi$ preserves the cone $P_\theta$ based on the maximum principle for the weak solution in \cite{22} and the assumption $f\geq 0$. Moreover, it satisfies $\varphi(0)=0$ since $f(z,0)\equiv 0$ according to condition (iii) of Theorem \ref{th1.5}.
\vspace{1mm}\par
Next, we verify the condition \textup{(i')} in Remark \ref{rem3}. By condition \textup{(iii)} of the theorem take proper $\widetilde{r}>0$ and $\alpha<\lambda_1$, such that
for $z \in \overline{\Omega}$:
\begin{equation}\label{f}
  f(z,t)\leq\alpha t,\;~\mbox{for}~0\leq t\leq\widetilde{r}.
\end{equation}
Define $A:=\alpha L$, then $A:P_\theta\rightarrow P_\theta$ is a bounded linear operator.
Moreover, according to the weak maximum principle \cite[Theorem 5.6]{22} and \eqref{f}, we have
\begin{equation}\label{001.7}
  \varphi(u)=Lf(z,u)\leq\alpha Lu=Au, ~\mbox{for}~\|u\|\leq r.
\end{equation}
Since the spectral radius of $A$ is strictly less than 1, the condition \textup{(i')} is satisfied.

Next we verify condition \textup{(ii')} in Remark \ref{rem3}. Define
\begin{equation}\label{001.8}
  F(u,t)=L(f(z,u)+t\phi_1),
\end{equation}
where $t\in(0,+\infty)$ and $\phi_1$ denotes an eigenfunction corresponding to $\lambda_1$. We may assume that $\|\phi_1\|_{L^2(\Omega)}=1$. We first verify \textup{(3)} in \textup{(ii')} by contradiction.
{Assume that there is a sequence $(t_j) \subset \mathbb{R}_+$ tending to infinity such that
$F(v_j,t_j)=v_j$ has a solution $v_j\in \overline{B}_R(0)$ for each $j \in \mathbb{N}$.} That is to say that
\begin{equation}\label{001.125}
  -\Delta_\gamma v_j=f(z,v_j)+t_j\phi_1
\end{equation}
has a solution $v_j$.
By condition \textup{(ii)} of the theorem, there are constants $M>0$ and
$k\in(\lambda_1,\displaystyle\lim_{t\rightarrow+\infty}\frac{f(z,t)}{t})$, such that
\begin{equation}\label{kM}
 f(z,t)\geq kt-M,\;~\mbox{for}~t\rightarrow+\infty.
\end{equation}
We multiply \eqref{001.125} by $\phi_1$ and integrate over $\Omega$ to obtain
\begin{equation}\label{001.10}
  -\int_\Omega\phi_1(z)\Delta_\gamma v_j(z)dz=\int_\Omega\phi_1(z)f(z,v_j(z))dz+t_j.
\end{equation}
Thus by \eqref{kM} and integration by parts, it holds that
\begin{equation}\label{dui}
  \lambda_1\int_\Omega\phi_1(z)v_j(z)dz\geq k\int_\Omega\phi_1(z)v_j(z)dz-M\int_\Omega\phi_1(z)dz+t_j,
\end{equation}
i.e.
\begin{equation}\label{dui1}
  t_j+(k-\lambda_1)\int_\Omega\phi_1(z)v_j(z)dz\leq M\int_\Omega\phi_1(z)dz\leq M|\Omega|^{\frac{1}{2}},
\end{equation}
where $|\Omega|$ denotes the measure of $\Omega$. Since $v_j\in \overline{B}_R(0)$  for each $j$ the inequality \eqref{dui1} implies that
\begin{equation}\label{dui2}
  t_j\leq C_1\,
\end{equation}
for some $C_1>0$ which contradicts $t_j\rightarrow+\infty$. Therefore, we conclude that $F(v,t)=v$ has no solution for $t\geq t_0$, where $t_0>0$ is sufficiently large.
\vspace{1mm}\par
Now we verify \textup{(1)} and \textup{(2)} in \textup{(ii')}.
For each choice of $t>0$, we apply Theorem \ref{th1.4} to the equation $-\Delta_\gamma v=f(z,v)+t\phi_1$. The new nonlinearity is $f+t\phi_1$, which is continuous. And for $p\in[1+\frac{4}{N_\gamma},\frac{N_\gamma+2}{N_\gamma-2})$, it still satisfies $\displaystyle\lim_{v\rightarrow+\infty}\frac{f(z,v)+t\phi_1}{v^p}=h(z)$. By Theorem \ref{th1.4}, if $v\in H^{1,2}_\gamma(\Omega)\cap C^\theta(\overline{\Omega})$ is any positive solution to $-\Delta_\gamma v=f(z,v)+t\phi_1$, then it holds that $\sup_{z\in\overline{\Omega}}v(z)\leq C_2$, where $C_2$ depends on $p$, $\Omega$, the behavior of $f+t\phi_1$ in the limit at infinity (i.e. the constant $C_2$ may depend on $h$, but does not depend on the value of $t$). Combined with properties of $f+t\phi_1$, we obtain $\|v\|_{H^{1,2}_\gamma(\Omega)}\leq C_3$. Thus there exists a constant $R$ large enough such that this equation has no solution when $\|v\|_{H^{1,2}_\gamma(\Omega)}\geq R$.
Therefore, the condition $F(v,t)\neq v$ for $\|v\|_{H^{1,2}_\gamma(\Omega)}=R$, $t\in(0,+\infty)$ is verified.
\vspace{1mm}\par
Furthermore, the condition $F(u,0)=\varphi(u)$ for $\|u\|_{H^{1,2}_\gamma(\Omega)}=R$ is easily verified by the definition of $F(u,t)$.
According to Proposition \ref{lemma2}, $\varphi$ has a fixed point in $U= \{ x \in P_\theta \: |\: r < \|x\| <R\}$, which means that Equation \eqref{002.1} has at least one positive solution $u\in H^{1,2}_\gamma(\Omega)\cap C^\theta(\overline{\Omega})$.
\end{proof}

\begin{center}
\noindent\textbf{Statements and Declarations}.~~
\end{center}
{\bf Data availability statement:}  The manuscript has no associated data.
\vspace{1mm}\\
{\bf Competing interests:} On behalf of all authors, the corresponding author states that there is no conflict of interest.
\vspace{1mm}\\
{\bf Funding:}
This work is supported by the NSFC under the grands 12271269 and the Fundamental Research Funds for the Central Universities.


\normalem

\end{document}